\pdfoutput=1 
\documentclass[10pt]{article} 

\usepackage{amsmath, amssymb}
\usepackage{amsthm} %
\usepackage{graphicx} 
\usepackage{enumitem} %

\newtheorem{THM}{Theorem}
\newtheorem{thm}{Theorem}[section]
\newtheorem{Lem}[thm]{Lemma}
\newtheorem{Prop}[thm]{Proposition}
\newtheorem{cor}[thm]{Corollary}
\theoremstyle{definition}
\newtheorem{Bsp}{Example}
\newtheorem*{Problem}{Problem}

\newtheorem*{thmgeneralk}{Theorem~\ref{generalk}}
\newtheorem*{thmmatroids}{Theorem~\ref{matroids}}
\newtheorem*{thmclusters}{Theorem~\ref{clusters}}

\newenvironment{txteq}
  {
    \begin{equation}
    \begin{minipage}[c]{0.85\textwidth} 
    \em                                
  }
  {\end{minipage}\end{equation}\ignorespacesafterend}
\newenvironment{txteq*}
  {
    \begin{equation*}
    \begin{minipage}[c]{0.85\textwidth} 
    \em                                
  }
  {\end{minipage}\end{equation*}\ignorespacesafterend}

\newcommand{\lek}{(<$$k)}
\newcommand{\defgl}{\mathrel{\mathop:}=}
\def\dcl(#1){\lceil #1 \rceil}
\newcommand{\pdffig}[4]{
   \begin{figure}[#4]
      \centering
      \includegraphics[width=#3\textwidth]{#1.pdf}
      \caption{#2}\label{pic:#1}
   \end{figure}
}

\def\vu{\kern-1pt\lowfwd u{1.5}1\kern-1pt}
\def\uv{\kern-1pt\lowbkwd u{1.5}1\kern-1pt}
\def\vw{\kern-1pt\lowfwd w{1.5}1\kern-1pt}
\def\wv{\kern-1pt\lowbkwd w{1.5}1\kern-1pt}

\def\vedash{{\mathop{\kern0pt e\lower.5pt\hbox{${}
      \scriptstyle'$}}\limits^{\kern0pt\raise.02ex
      \vbox to 0pt{\hbox{$\scriptscriptstyle\rightarrow$}\vss}}}}
\def\evdash{{\mathop{\kern0pt e\lower.5pt\hbox{${}
      \scriptstyle'$}}\limits^{\kern0pt\raise.02ex
      \vbox to 0pt{\hbox{$\scriptscriptstyle\leftarrow$}\vss}}}} 

\def\lowfwd #1#2#3{{\mathop{\kern0pt #1}\limits^{\kern#2pt\raise.#3ex
\vbox to 0pt{\hbox{$\scriptscriptstyle\rightarrow$}\vss}}}}
\def\lowbkwd #1#2#3{{\mathop{\kern0pt #1}\limits^{\kern#2pt\raise.#3ex
\vbox to 0pt{\hbox{$\scriptscriptstyle\leftarrow$}\vss}}}}

\def\ve{\kern-1pt\lowfwd e{1.5}2\kern-1pt}

\def\vedash{{\mathop{\kern0pt e\lower.5pt\hbox{${}
     \scriptstyle'$}}\limits^{\kern0pt\raise.02ex
     \vbox to 0pt{\hbox{$\scriptscriptstyle\rightarrow$}\vss}}}}
\def\evdash{{\mathop{\kern0pt e\lower.5pt\hbox{${}
     \scriptstyle'$}}\limits^{\kern0pt\raise.02ex
     \vbox to 0pt{\hbox{$\scriptscriptstyle\leftarrow$}\vss}}}}
\def\vf{\kern-1pt\lowfwd f{1.5}2\kern-1pt}

\def\vr{\lowfwd r{1.5}2}
\def\rv{\lowbkwd r02}

\def\vri{\lowfwd {r_i}12}

\def\vrj{\lowfwd {r_j}12}

\def\rvk{\lowbkwd {r_k}02}
\def\vrdash{{\mathop{\kern0pt r\lower.5pt\hbox{${}
     \scriptstyle'$}}\limits^{\kern0pt\raise.02ex
     \vbox to 0pt{\hbox{$\scriptscriptstyle\rightarrow$}\vss}}}}
\def\rvdash{{\mathop{\kern0pt r\lower.5pt\hbox{${}
     \scriptstyle'$}}\limits^{\kern0pt\raise.02ex
     \vbox to 0pt{\hbox{$\scriptscriptstyle\leftarrow$}\vss}}}}
\def\vrddash{{\mathop{\kern0pt r\lower.5pt\hbox{${}
     \scriptstyle''$}}\limits^{\kern0pt\raise.02ex
     \vbox to 0pt{\hbox{$\scriptscriptstyle\rightarrow$}\vss}}}}
\def\rvddash{{\mathop{\kern0pt r\lower.5pt\hbox{${}
     \scriptstyle''$}}\limits^{\kern0pt\raise.02ex
     \vbox to 0pt{\hbox{$\scriptscriptstyle\leftarrow$}\vss}}}}
\def\vrone{\lowfwd {r_1}02}

\def\vrtwo{\lowfwd {r_2}12}

\def\vrthree{\lowfwd {r_3}12}
\def\rvthree{\lowbkwd {r_3}02}
\def\vR{{\vec R}} 
\def\vs{\lowfwd s{1.5}1}
\def\sv{\lowbkwd s{1.5}1}
\def\vsdash{{\mathop{\kern0pt s\lower.5pt\hbox{${}
     \scriptstyle'$}}\limits^{\kern0pt\raise.02ex
     \vbox to 0pt{\hbox{$\scriptscriptstyle\rightarrow$}\vss}}}}
\def\svdash{{\mathop{\kern0pt s\lower.5pt\hbox{${}
     \scriptstyle'$}}\limits^{\kern0pt\raise.02ex
     \vbox to 0pt{\hbox{$\scriptscriptstyle\leftarrow$}\vss}}}}

\def\vsidash{{\mathop{\kern0pt s_i\kern-3.5pt\lower.3pt\hbox{${}
     \scriptstyle'$}}\limits^{\kern0pt\raise.02ex
     \vbox to 0pt{\hbox{$\scriptscriptstyle\rightarrow$}\vss}}}}
\def\vS{{\vec S}} 
\def\vSr{{\vec S}_{\raise.1ex\vbox to 0pt{\vss\hbox{$\scriptstyle\ge\vr$}}}}
\def\vSrone{{\vec S}_{\raise.1ex\vbox to 0pt{\vss\hbox{$\scriptstyle\ge\vrone$}}}}
\def\vSdash{{\mathop{\kern0pt S\lower-1pt\hbox{${}
     \scriptstyle'$}}\limits^{\kern2pt\raise.1ex
     \vbox to 0pt{\hbox{$\scriptscriptstyle\rightarrow$}\vss}}}}
\def\vt{\lowfwd t{1.5}1}
\def\tv{\lowbkwd t{1.5}1}
\def\vtdash{{\mathop{\kern0pt t\lower-1pt\hbox{${}
     \scriptstyle'$}}\limits^{\kern0pt\raise.02ex
     \vbox to 0pt{\hbox{$\scriptscriptstyle\rightarrow$}\vss}}}}
\def\tvdash{{\mathop{\kern0pt t\lower-1pt\hbox{${}
     \scriptstyle'$}}\limits^{\kern-1pt\raise.02ex
     \vbox to 0pt{\hbox{$\scriptscriptstyle\leftarrow$}\vss}}}}
\def\vT{{\vec T}} 
\def\vTdash{{\mathop{\kern0pt T\lower-1pt\hbox{${}
     \scriptstyle'$}}\limits^{\kern2pt\raise.1ex
     \vbox to 0pt{\hbox{$\scriptscriptstyle\rightarrow$}\vss}}}}
\def\vTkdash{{\mathop{\kern0pt T_k\kern-3pt\lower-2pt\hbox{${}%
     \scriptstyle'$}}\limits^{\kern2pt\raise.1ex
     \vbox to 0pt{\hbox{$\scriptscriptstyle\rightarrow$}\vss}}}}
\def\vU{{\vec U}\!} 
\def\vx{\lowfwd x{1.5}2}

\def\vy{\lowfwd y{1.5}2}

\def\sub{\subseteq}
\def\supe{\supseteq}
\def\sm{\smallsetminus}
\def\td{tree-decom\-po\-si\-tion}
\def\restricts{\!\restriction\!}

\def\O{\mathcal{O}}
\def\P{\mathcal{P}}
\def\PT{P_{T'}}
\def\T{\mathcal{T}}
\def\V{\mathcal{V}}

\newcommand\COMMENT[1]{}

\begin{document}

\title{Profiles of separations:\\ in graphs, matroids, and beyond}
\author{Reinhard Diestel \and Fabian Hundertmark\and Sahar Lemanczyk\\ [10pt]
  Mathematisches Seminar, Universit\"at Hamburg}

\maketitle

\begin{abstract}\noindent
  We show that all the tangles in a finite graph or matroid can be distinguished by a single \td\ that is invariant under the automorphisms of the graph or matroid. This comes as a corollary of a similar decomposition theorem for more general combinatorial structures, which has further applications.

These include a new approach to cluster analysis and image segmentation. As another illustration for the abstract theorem, we show that applying it to edge-tangles yields the Gomory-Hu theorem.
   \end{abstract}

\section{Introduction}\label{sec:intro}

One of the oldest, and least precise, problems in the connectivity theory of graphs is how to decompose a $(k-1)$-connected graph into something like its `$k$-connected pieces'. The block-cutvertex tree of a graph achieves this for $k=2$, and Tutte~\cite{TutteGrTh} showed for $k=3$ that every 2-connected finite graph has a \td\ of adhesion~2 whose torsos are either 3-connected or cycles. Tutte's theorem has recently been extended by Grohe~\cite{GroheQuasi4} to $k=4$, and by Carmesin, Diestel, Hundertmark and Stein~\cite{confing} to arbitrary~$k$ as follows.

A \emph{k-block} of a graph $G$, where $k$ is any positive integer, is a maximal set $X$ of at least $k$ vertices such that no two vertices $x,y\in X$ can be separated in $G$ by fewer than $k$ vertices other than $x$ and~$y$. We refer to $k$-blocks for unspecified~$k$ simply as {\em blocks}.

Tutte's theorem is essentially the case $k=3$ of the following more recent result. Note however that Theorem~\ref{kblocks} does not require the graph to be $(k-1)$-connected.

\begin{THM}\label{kblocks}
   \emph{\cite[Theorem 1]{confing}} Every finite graph has a canonical \td{} 
of adhesion $<k$ that efficiently distinguishes all its $k$-blocks, for every $k\in\Bbb N$.
\end{THM}

Here, a \td{} $(T,\mathcal V)$  of a graph $G$ with $\mathcal V = (V_t)_{t \in T}$ is {\em canonical\/} if the automorphisms of~$G$ act naturally on~$T$: if they map every part $V_t$ to another part~$V_{t'}$ in such a way that $t\mapsto t'$ is an automorphism of~$T$. It \emph{distinguishes} two sets $b,b^\prime$ of vertices in~$G$ if $T$ has an edge~$e$ whose associated separation of~$G$ can be written as $\{A,B\}$ with $b\sub A$ and $b'\sub B$. (The edge $e=t_1 t_2\in T$ is {\em associated with\/} the separation $\{\bigcup_{t\in T_1} V_t\,,\bigcup_{t\in T_2} V_t\}$ of~$G$ where $T_i$ is the component of $T-e$ containing~$t_i$, and the least order%
   \footnote{The {\em order\/} of a separation $\{A,B\}$ of a graph is the number~$|A\cap B|$. Separations of order~$k$ are $k$-{\em separations\/}; separations of order~$<k$ are $(<k)$-{\em separations\/}.}
   of such a separation is the {\em adhesion\/} of $(T,\mathcal V)$.)
It distinguishes $b$ and~$b'$ {\em efficiently} if $e$ can be chosen so that $\{A,B\}$ has the least order~$\ell$ of any separation $\{C,D\}$ of $G$ such that $b\sub C$ and $b'\sub D$. It is not hard to show that $\ell < k$ if $b$ and~$b'$ are $k$-blocks.

More generally, given a $k$-block $b$ and a {\em $\lek$-separation\/} $\{A,B\}$ of~$G$, there cannot be vertices $x,y$ in $b$ such that $x \in A\sm B$ and $y \in B \sm A$, as these vertices would then be separated by the set $A \cap B$ of fewer than $k$ vertices. So $b$ will be a subset of $A$ or of~$B$, but not both. In this way, $b$~{\em orients\/} $\{A,B\}$: as $(A,B)$ 
if $b \subseteq B$, and as $(B,A)$ if $b \subseteq A$. (We think of an {\em oriented separation} $(A,B)$ as `pointing towards'~$B$.)

Every $k$-block $b$ of $G$ thus induces an {\em orientation\/} of the set $S_k$ of all the $\lek$-separations of~$G$: a set consisting of exactly one of the two orientations of every separation in~$S_k$. We call this orientation of $S_k$ the \emph{$k$-profile} of~$b$.

An analysis of the proof of Theorem~\ref{kblocks} shows that all it ever uses about $k$-blocks is their $k$-profiles. And the theorem itself can be expressed more easily in these terms too. Indeed, it finds a canonical \td\ $(T,\mathcal V)$ of the given graph that {\em distinguishes\/} the $k$-profiles of its $k$-blocks in the following natural sense: for every two such profiles $P$ and~$P'$ there is a separation $\{A,B\}\in S_k$ associated with an edge of~$T$ that has orientations $(A,B)\in P$ and $(B,A)\in P'$.\looseness=-1

The $k$-profiles of the $k$-blocks of a graph are not unlike its $k$-tangles: these, too, are orientations of~$S_k$.%
   \footnote{Indeed, in Section~\ref{secProfiles} we shall introduce more abstract `$k$-profiles', not necessarily induced by $k$-blocks, of which $k$-tangles~-- those of order~$k$, in the terminology of~\cite{GMX}~-- are a prime example. See~\cite{ForcingBlocks, CDHH13CanonicalAlg} for more on the relationship between profiles, blocks and tangles.}
   Indeed, the following generalization of Theorem~\ref{kblocks} is a corollary merely of rewriting its proof in terms of profiles:

\begin{THM}\emph{\cite[Theorem 4.5]{CDHH13CanonicalAlg}}\label{fixedk}
Every finite graph has a canonical \td{} of adhesion $<k$ that efficiently distinguishes all its $k$-blocks and $k$-tangles, for every fixed $k\in\Bbb N$.
\end{THM}

Robertson and Seymour~\cite{GMX} proved that every finite graph has a \td{} that distinguishes all its distinguishable tangles, not just all those of some fixed order~$k$. One of our main results is that we obtain a canonical such decomposition. Using profiles as a common generalization of blocks and tangles, we can even distin\-guish both at the same time, and from each other, by the same decomposition:

\begin{THM}\label{generalk}
Every finite graph has a canonical \td{} that efficiently distinguishes all its distinguishable tangles%
   \COMMENT{}
   and robust\/%
   \footnote{This is a necessary but unimportant restriction; all interesting blocks are robust (Section~\ref{subsec:robust}).}\,%
   blocks.
\end{THM}

\noindent
   See Section~\ref{subsec:graphs} for precise definitions.

Representing the $k$-blocks of a graph~$G$ by their $k$-profiles represents a shift of paradigm for our initial question about those `$k$-connected components': $k$-profiles can identify highly connected regions of~$G$ without being concrete objects, such as subgraphs or blocks, simply by pointing to them. In this new paradigm, we would think of these orientations of~$S_k$ as the `abstract objects' we wish to distinguish by some \td.

Not all orientations of $S_k$ identify highly connected regions of~$G$, though: they must be consistent at least in the sense that no two oriented separations point away from each other. Profiles of blocks are naturally consistent in this sense, and tangles are too, by definition. Other consistent orientations that identify highly connected regions in a graph are explored in~\cite{ProfileDuality, TangleTreeGraphsMatroids}.

Another advantage of this new paradigm is that we can study, and seek to separate in a tree-like way, highly connected substructures in combinatorial structures other than graphs: all we need is a sensible%
   \COMMENT{}
   notion of separation. Tangles for matroids are an example of this, and we obtain the following analogue of Theorem~\ref{generalk}:

\begin{THM}\label{matroids}
   Every finite matroid has a canonical \td{} which
   efficient\-ly distinguishes all its distinguishable tangles.%
   \COMMENT{}
\end{THM}

\noindent
   See Section~\ref{subsec:matroids} for definitions.

Theorem~\ref{matroids} generalizes the matroid analogue of the theorem of Robertson and Seymour mentioned earlier: Geelen, Gerards and Whittle~\cite{TanglesInMatroids} proved that all the distinguishable tangles of a matroid can be distinguished by a single \td{}. But while their \td s are again not canonical, ours are.

\medbreak

One main aspect of this paper is that tangle-like structures are meaningful, and can be canonically separated in a tree-like way, in a much more general context even than graphs and matroids combined. Just as, for proving our theorems, we do not need to know more about $k$-blocks than how they orient the separations in~$S_k$, all we need to know about these separations is how they and their orientations are, or fail to be, nested and consistent with each other. This information, however, can be captured by a simple poset with an order-reversing involution defined on the set of these separations: we do not need that they `separate' anything, such as the ground set of a matroid or the vertex set of a graph, into two sides.

We shall prove our theorems at this general level of {\em abstract separation systems\/}~\cite{AbstractSepSys}: we define {\em profiles\/} as consistent ways of orienting these abstract separations, and find in any separation system~$S$ a nested subset $T$ of separations that {\em distinguish\/} all the profiles of~$S$,%
   \COMMENT{}
   in that for every pair of distinct profiles there exists a separation in $T$ which they orient differently~-- just as the profiles of two blocks would if these lay on either side of a separation which they both orient.

The fact that we our main theorems can be proved in this abstract setting makes them applicable beyond graphs and matroids. For example, our  {\it canonical tangle-tree theorem\/}, Theorem~\ref{thm2}, of which Theorems~\ref{generalk} and~\ref{matroids} will be corollaries, can also be applied to image segmentation and cluster analysis in big data sets. In that context, clusters are described as profiles of bespoke separation systems for the data set in question, which may have nothing to do with graphs or matroids.

This approach to cluster analysis is new. It takes advantage of the fact that profiles, like real-world clusters, can be `fuzzy'. For example, consider a large grid in a graph. For every low-order separation, most of the grid will lie on the same side, so the grid `orients' that separation towards this side. But every single vertex will lie on the `wrong' side for {\em some\/} low-order separation, the side not containing most of the grid; for example, it may be separated off by its four neighbours. The grid, therefore, defines a unique $k$-profile for some large~$k$, but the `location' of this profile is not represented correctly by any one of its vertices~-- just as for a fuzzy cluster of a data set it may be impossible to say which data exactly belong to that cluster and which do not.

Profiles of abstract separation systems can capture such clusters, and our theorems can be applied to describe their relative positions, as soon as the data set comes with a submodular `order' function on its separations. In practice, it appears that most natural ways of cutting a data set in two allow for such order functions, and their choice can be used to specify the exact type of cluster to be analysed. An example from image segmentation is described in~\cite{MonaLisa}.

The application of our main result to cluster analysis in large data sets will read as follows:

\begin{THM}\label{clusters}
Every submodular data set has a canonical regular tree set of separations which efficiently distinguishes all its distinguishable clusters.
\end{THM}

Abstract separation systems were first introduced in~\cite{AbstractSepSys}. In Section~\ref{secProfiles} we give a summary of the basic formal setup, and define profiles as their orientations: a~choice of one element from each pair identified by the involution. As our main results, we then prove in Section~\ref{secTT} two general tangle-tree theorems for these abstract separation systems: one canonical, the other not but best possible in other ways. In Section~\ref{sec:applications} we then apply the first of these to graphs and matroids, obtaining Theorem~\ref{generalk} and Theorem~\ref{matroids} as corollaries. 

In Section~\ref{sec:furtherapplications} we first show how to set up abstract separation systems for cluster analysis in large data sets, and show how Theorem~\ref{clusters} follows from our canonical tangle-tree theorem. We then illustrate the non-canonical tangle-tree theorem by applying it to the `edge-tangles' of a graph. This yields the existence of Gomory-Hu trees known from optimization, reproving the well known Gomory-Hu theorem.

There is also a width-duality theorem for profiles in graphs and matroids, like those for tangles known from~\cite{GMX}: if a given graph, say,%
   \COMMENT{}
   has no $k$-profile for some given~$k$, or no $k$-profile induced by a $k$-block,%
   \COMMENT{}
   it admits a \td\ witnessing this as an easily checked certificate. This duality theorem will be proved elsewhere~\cite{ProfileDuality} but is indicated briefly in Section~\ref{sec:duality}.

We conclude in Section~\ref{sec:problems} with an open problem for tangles.

Any graph-theoretic terms not defined here can be found in \cite{DiestelBook16}. For matroid terminology we refer to~\cite{OxleyBook}. More on general abstract separation systems can be found in~\cite{AbstractSepSys}, and more on tree sets in~\cite{TreeSets}. This paper includes and supersedes~\cite{profiles}.

\section{Abstract separation systems and their profiles}\label{secProfiles}

\subsection{Separations of sets}\label{subsec:basics}

Separations in graphs and matroids are certain unordered pairs of subsets of a set: of the vertex set of the graph or the ground set of the matroid. In the case of graphs, the two subsets may overlap; 
in the case of a matroid, they partition the ground set. But in either case their union is the original set.

Given an arbitrary set~$V\!$, a {\em separation\/} of $V$ is a set $\{A,B\}$ of two subsets $A,B$ such that $A\cup B = V\!$. Its {\em order\/} is the cardinality $|A\cap B|$. Every such separation $\{A,B\}$ has two \emph{orientations}: $(A,B)$ and $(B,A)$. 
Inverting these is an involution $(A,B) \mapsto (B,A)$ on the set of these \emph{oriented separations} of~$V\!$.

The oriented separations of a set are partially ordered as
\begin{equation*}
(A,B) \leq (C,D) :\Leftrightarrow A \subseteq C\text{ and } B\supseteq D.
\end{equation*}
Our earlier involution reverses this ordering:
\begin{equation*}
(A,B)\leq (C,D) \Leftrightarrow (B,A)\ge (D,C).
\end{equation*}

The oriented separations of a set form a lattice under this partial ordering, in which $(A\cap C, B\cup D)$ is the infimum of $(A,B)$ and~$(C,D)$, and $(A\cup C, B\cap D)$ is their supremum. The infima and suprema of two separations of a graph or matroid are again separations of that graph or matroid, so these too form lattices under~$\le$.

Their induced posets of all the oriented separations of order~$<k$ for some fixed~$k$, however, need not form such a lattice: when $(A,B)$ and $(C,D)$ have order~$<k$, this need not be the case for $(A\cap C, B\cup D)$ and~$(A\cup C, B\cap D)$.

\subsection{Abstract separation systems}\label{subsec:SeparationSystems}

A {\em separation system\/} $(\vS,\le\,,\!{}^*)$ is a partially ordered set $\vS$ with an order-reversing involution${\,}^*$.  An {\em isomorphism\/} between two separation systems is a bijection between their underlying sets that respects their partial orderings and commutes with their involutions.

The elements of a separation system~$\vS$ are called {\em oriented separations\/}. When a given element of $\vS$ is denoted as~$\vs$, its {\em inverse\/}~$\vs^*$ will be denoted as~$\sv$, and vice versa. The assumption that${\,}^*$ be {\em order-reversing\/} means that, for all $\vr,\vs\in\vS$,
\begin{equation}\label{invcomp}
\vr\le\vs\ \Leftrightarrow\ \rv\ge\sv.
\end{equation}
For subsets $R\sub\vS$ we write $R^*:=\{\,\rv\mid \vr\in R\,\}$.

An (unoriented) {\em separation\/} is a set of the form $\{\vs,\sv\}$, and then denoted by~$s$.%
   \footnote{To smooth the flow of the narrative we usually also refer to oriented separations simply as `separations' if the context or use of the arrow notation~$\vs$ shows that they are oriented.}
   We call $\vs$ and~$\sv$ the {\em orientations\/} of~$s$. The set of all such sets $\{\vs,\sv\}\sub\vS$ will be denoted by~$S$. If $\vs=\sv$, we call both $\vs$ and $s$ {\em degenerate\/}.%
   \COMMENT{}

When a separation is introduced ahead of its elements and denoted by a single letter~$s$, its elements will then be denoted as $\vs$ and~$\sv$.%
   \footnote{It is meaningless here to ask which is which: neither $\vs$ nor $\sv$ is a well-defined object just given~$s$. But given one of them, both the other and $s$ will be well defined.}
   Given a set $R\sub S$ of separations, we write $\vR := \bigcup R\sub\vS$%
   \COMMENT{}
   for the set of all the orientations of its elements. With the ordering and involution induced from~$\vS$, this is again a separation system.%
   \COMMENT{}

A separation $\vr\in\vS$ is {\em trivial in~$\vS$\/}, and $\rv$ is {\em co-trivial\/}, if there exists $s \in S$ such that $\vr < \vs$ as well as $\vr < \sv$.%
   \COMMENT{}
   We call such an $s$ a {\em witness\/} of $\vr$ and its triviality. If neither orientation of~$r$ is trivial, we call~$r$ {\em nontrivial\/}.%
   \COMMENT{}
    Note that if $\vr$ is trivial in~$\vS$ then so is every $\vrdash \le \vr$.

A separation $\vs$ is {\em small\/} if $\vs\le\sv$. Trivial separations are small by~\eqref{invcomp}, but other separations can be small too. But if $\vs$ is small and $\vr < \vs$, then $\vr$ is clearly trivial.%
   \COMMENT{}
   So all but the largest small separations are in fact trivial. An unoriented separation is {\em proper\/} if it has no small orientation. A~separation system is {\it regular\/} if none of its elements is small~\cite{AbstractSepSys}.

If there are binary operations $\vee$ and~$\wedge$ on a separation system $(\vU,\le\,,\!{}^*)$ such that $\vr\vee\vs$ is the supremum and $\vr\wedge\vs$ the infimum of $\vr$ and~$\vs$ in~$\vU$, we call $(\vU,\le\,,\!{}^*,\vee,\wedge)$ a {\em universe\/} of (oriented) separations. By~\eqref{invcomp}, it satisfies De~Mor\-gan's law:
\begin{equation}\label{deMorgan}
   (\vr\vee\vs)^* =\> \rv\wedge\sv.
\end{equation}%
   \COMMENT{}

The universe~$\vU$ is {\em submodular\/}%
   \COMMENT{}
   if it comes with a submodular \emph{order function}, a real function $\vs\mapsto |\vs|$ on~$\vU$%
   \COMMENT{}
   that satisfies $0\le |\vs| = |\sv|$ and 
 $$|\vr\vee\vs| + |\vr\wedge\vs|\le |\vr|+|\vs|$$
for all $\vr,\vs\in \vU$. We call $|s| := |\vs|$ the \emph{order} of $s$ and of~$\vs$. For every integer~$k>0$, then,
 $$\vS_k := \{\,\vs\in \vU : |\vs| < k\,\}$$
is a separation system. But $\vS_k$ need not itself be a universe, at least not with respect to the operations $\lor$~and~$\land$ induced from~$\vU$, since suprema or infima in~$\vU$ of elements of~$\vS_k$ can lie outside~$\vS_k$.%
   \COMMENT{}

An {\em isomorphism\/} between universes of separations is a bijection between their ground sets that respects their orderings%
   \COMMENT{}
   and commutes with their involutions and, in the case of submodular universes, their order functions.

Separations of a set~$V\!$, and their orientations, are clearly an instance of this abstract setup if we identify $\{A,B\}$ with $\{(A,B),(B,A)\}$. The small separations of~$V$ are those of the form~$(A,V)$, the trivial ones are those of the form $(U,V)$ with $U\sub A\cap B$ for some other separation $\{A,B\}\ne\{U,V\}$, and the proper separations are those of the form $\{A,B\}$ with $A\sm B$ and $B\sm A$ both nonempty.
The separations of~$V\!$ form a submodular universe for the order function $|(A,B)| := |A,B| := |A\cap B|$.

\subsection{Tree sets of separations}\label{subsec:nested}

Given a separation system $(\vS,\le\,,\!{}^*)$, two separations $r,s\in S$ are {\em nested\/} if they have comparable orientations; otherwise they \emph{cross}. Two oriented separations $\vr,\vs$ are {\em nested\/} if $r$ and~$s$ are nested.%
   \footnote{Terms introduced for unoriented separations may be used informally for oriented separations too if the meaning is obvious, and vice versa.}%
   \COMMENT{}
   We say that $\vr$ {\em points towards\/}~$s$, and $\rv$ {\em points away from\/}~$s$, if $\vr\le\vs$ or $\vr\le\sv$. Then two nested oriented separations are either comparable, or point towards each other, or point away from each other. 

A~set of separations is {\em nested\/} if every two of its elements are nested. Two sets of separations are {\em nested\/} if every element of the first set is nested with every element of the second. A {\em tree set\/} is a nested separation system without trivial or degenerate elements.%
   \COMMENT{}
   When $\vT\sub\vS$ is a tree set, we also call $T\sub S$ a {\em tree set\/} (and {\em regular\/} if~$\vT$ is regular).

For example, the set of orientations $(u,v)$ of the edges $uv$ of a tree~$T$ form a regular tree set with respect to the involution $(u,v)\mapsto (v,u)$ and the \emph{natural partial ordering} on~$\vec E(T)$: the ordering in which $(x,y) < (u,v)$ if $\{x,y\}\ne\{u,v\}$%
   \COMMENT{}
   and the unique $\{x,y\}$--$\{u,v\}$ path in $T$ joins $y$ to~$u$.
   The oriented bipartitions of~$V(T)$ defined by deleting an edge of~$T$ form a tree set isomorphic to this one. 

The separations of a graph associated with a \td\ are also nested. But since they may be small, or even trivial, they need not form a tree set, not even an irregular one. But conversely, every tree set%
   \COMMENT{}
   of separations of a finite graph~$G$ comes from a \td\ of~$G$, which is essentially unique~\cite{TreeSets}.%
   \COMMENT{}

Any two elements $r,s$ of a universe $\vS$ of separations have four \emph{corner separations}
 $$\{(\vr \lor \vs) ,(\vr \lor \vs)^\ast \}\>,\ \{(\rv \lor \sv) ,(\rv \lor \sv)^\ast \}\>,\ \{(\vr \lor \sv) ,(\vr \lor \sv)^\ast \}\>,\ \{(\rv \lor \vs) , (\rv \lor \vs)^\ast \}$$
(see Figure~\ref{fig:nontrans}). By~\eqref{deMorgan}, these can also be expressed in terms of~$\wedge$. 

   \begin{figure}[htpb]
\centering
        \includegraphics{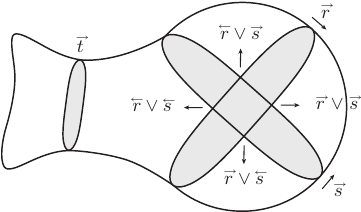}
        \caption{Separations as in Lemma~\ref{lem:fish}}
   \label{fig:nontrans}\vskip-9pt\vskip0pt
   \end{figure}

\begin{Lem} \label{lem:fish}
   Let $r,s \in \vS{}$ be two crossing separations. Every separation $t$ that is nested
   with both $r$ and $s$ is also nested with all four corner separations of $r$ and $s$.
\end{Lem}

\begin{proof}
Since $t$ is nested with $r$ and~$s$, it has an orientation pointing towards~$r$, and one pointing towards~$s$. If these orientations of $t$ are not the same, then $\vr\le\vt\le\vs$ for suitable orientations of $r,s,t$. In particular, $r$ and $s$ are nested, contrary to our assumption. Hence $t$ has an orientation $\vt$ that points towards both $r$ and~$s$.

Now $r$ and~$s$ have orientations $\vt\le\vr$ and $\vt\le\vs$. Since $\wedge$ and $\vee$ denote infima and suprema in~$\vS$, we have $\vt\le\vr\land\vs = (\rv\lor\sv)^*$ by~\eqref{deMorgan}, as well as trivially $\vt\le\vr\lor\vs$ and $\vt\le\vr\lor\sv$ and $\vt\le\rv\lor\vs$.
   \end{proof}

\subsection{Profiles of separation systems}\label{subsec:profiles}

Given a separation system $(\vS,\le\,,\!{}^*)$, a subset $O\sub\vS$ is an \emph{orientation} of $S$ (and of~$\vS$) if $O \cup O^\ast = \vS$ and $\lvert O\cap \{\vs,\sv \} \rvert = 1$ 
for all $s \in S$. Thus, $O$~contains every degenerate separation from~$\vS$ and contains exactly one orientation of every nondegenerate one. For subsets $S'\sub S$ we say that $O$ {\em induces\/} and {\em extends\/} the orientation $O\cap\vSdash$ of~$S'$, and thereby {\em orients\/}~$S'$.

A set $O\sub\vS$ is {\em consistent\/} if there are no distinct $r,s \in S$ with orientations $\vr < \vs$ such that $\rv, \vs \in O$. Consistent orientations of~$\vS$ contain all its trivial separations: if $\vr$ is trivial in~$\vS$, witnessed by~$s$ say, then $s$ cannot be oriented consistently with~$\rv$. Since $s$ must be oriented somehow, this implies that $\rv\notin O$ and hence $\vr\in O$.

Assume now that $(\vS,\le\,,\!{}^*)$ lies inside a universe $(\vU,\le\,,\!{}^*,\vee,\wedge)$.%
   \COMMENT{}
 Generalizing our notion of the profiles of blocks in graphs from the introduction, let us call an orientation $P$ of~$S$ a \emph{profile} (of~$S$ or~$\vS$) if it is consistent  and satisfies
 $$\text{For all $\vr,\vs \in P$ the separation $\rv \land \sv = (\vr\lor\vs)^*$ is not in $P$.}\eqno\rm(P)$$
Thus if $P$ contains $\vr$ and~$\vs$ it also contains $\vr\lor\vs$, unless $\vr\lor\vs\notin\vS$.%
  \COMMENT{}

A profile is {\em regular\/} if it contains no separation whose inverse is small. Examples of irregular profiles are those we call {\em special\/}: profiles $P$ of~$S$ for which there exists a separation $\vs\in P$ such that $\sv$ is small and
  $$P = \{\,\vr\in\vS\mid\vr\le\vs\,\}\sm\{\sv\}.$$%
   \COMMENT{}
   In general, $S$~can have irregular profiles that are not special, but this will not happen in the cases relevant to us.%
   \COMMENT{}

Applying (P) to any degenerate separation shows that if $S$ has a degenerate element it has no profile.%
   \COMMENT{}

Note that every subset $Q$ of a profile of~$\vS$ is a profile of~$\vR = Q\cup Q^*\sub\vS$.%
   \COMMENT{}
   Put another way, if $P$ is a profile of~$S$\vadjust{\penalty-200} and $R\sub S$, then $P\cap\vR$ is a profile of~$R$, which we say is {\em induced\/} by~$P$. If $\P$ is a set of profiles of~$S$, we write
 $$\P\restricts R := \{\,P\cap \vR\mid P\in\P\,\}$$
 for the set of profiles they induce on~$R$, and say that $\P$ {\em induces\/}~$\P\restricts R$.

For example, the $k$-profile of a $k$-block $b$ in a graph~$G$, as defined in the introduction, is the orientation $\{\,(A,B)\in\vS_k\mid b\sub B\,\}$ of the set $S_k$ of separations of~$G$ of order~$<k$. This orientation of~$S_k$ is consistent and satisfies~(P), and is thus a (regular) profile of~$S_k$. Similarly, every $k$-tangle of $G$ is a regular profile of~$S_k$; see Section~\ref{subsec:graphs}.

When $\vU$ is a universe of separations, we call the profiles of the separation systems $\vS_k\sub\vU$ of its separations of order~$<k$ the {\em $k$-profiles in\/}~$\vU$.%
   \COMMENT{}
   Even if we are interested only in the $k$-profiles for some fixed~$k$, it is usually important to have~$\vU$, rather than just~$\vS_k$, as a larger universe in which $\lor$ and~$\land$ are defined. Here is an example:

\begin{Bsp}\label{ex:resolution}
Let $G$ be a graph with a \td\ whose tree is a 3-star. Let the central part in this decomposition be a triangle $xyz$ and the other three parts pendant edges $xx'$, $yy'$ and~$zz'$. Let $\vr,\vs,\vt$ be the separations splitting $G$ at $x,y,z$ respectively, i.e., $\vr = (\{x',x\}, \{x,y,z,y',z'\})$, and similarly $\vs$ for $y$ and $\vt$ for~$z$. Then $P=\{\vr,\vs,\vt\}$, together with all the small separations of the form~$(v,V)$, is a 2-profile of~$G$ in the universe of all its separations, which points to its central triangle. Indeed, $P$~satisfies~(P) vacuously because, for example, $\vr\lor\vs$ is a separation of order~2, which does not get oriented by~$P$.

However if the universe in which $\lor$ and~$\land$ are defined was just the set $S = \{r,s,t\}$ of all the proper 1-separations of~$G$, then $\vr\lor\vs = \tv$ would violate~(P). In this sparser universe, the only 2-profiles would be the orientations of~$S$ orienting all its separations towards one of the leaves $x',y',z'$: it would lack the `resolution' for its profiles to see the central triangle as a highly connected substructure.
\end{Bsp}

Let us close this general section on profiles by showing that even in contexts where irregular profiles do occur, they are often just the special ones:

\begin{Prop}
If $\vS$ contains all trivial separations in~$\vU$ that have a witness in~$S$, then every irregular profile $P$ of~$S$ is special.
\end{Prop}

\begin{proof}
Since $P$ is irregular, it contains a separation $\vs$ whose inverse $\sv$ is small. If $P$ is not special, it contains another separation $\vr\not\le\vs$.%
   \COMMENT{}
   Then also $\vs\not < \vr$, since otherwise $\rv < \sv \le \vs$ would be trivial and hence in~$P$.%
   \COMMENT{}
   Hence $\vr,\vs < \vr\lor\vs$.%
   \COMMENT{}
   But now $(\vr\lor\vs)^* < \sv\le\vs$ is trivial, hence in~$S$ by assumption, and thus also in~$P$. But this means that $P$ violates~(P), a contradiction.
\end{proof}

\section{Tangle-tree theorems for profiles}\label{secTT}
 
As explained in the introduction, the paradigm proposed in this paper is that certain consistent orientations of a separation system~$\vS$, such as profiles, can be treated like highly connected substructures of some structure which these separations separate,\vadjust{\penalty-200} even if $\vS$ is given abstractly and no such structure is known. We shall therefore be interested in finding small subsets $T$ of $S$ that `distinguish' some given profiles of~$S$, in the following natural sense.

We say that a non-degenerate separation $s\in S$ \emph{distinguishes} two orientations $O,O^\prime$ of subsets of~$S$ if $s$ has orientations $\vs \in O$ and $\sv \in O'$. (We then also say that $\vs$ and~$\sv$ themselves distinguish $O$ from~$O'$.) The sets $O,O'$ are {\em distinguishable\/} if there exists some $s\in S$ that distinguishes them. A~set $T\sub S$ {\em distinguishes}~$O$ from~$O'$ if some $s\in T$ distinguishes them, and $T$ {\em distinguishes\/} a set~$\O$ of orientations of subsets of~$S$ if it distinguishes its elements pairwise.

A~separation $r$ and its orientations are $\O${\em -relevant} if $r$ distinguishes some two elements of~$\O$. A~set of separations is $\O$-{\em relevant\/} if all its elements are.

Later, these sets $\O$ will typically be sets~$\P$ of profiles: either of profiles of some fixed separation system~$\vS$, or of profiles in some fixed universe~$\vU$ of separations. (Note that a separation $r\in U$ can distinguish an $\ell$-profile in~$\vU$ from a $k$-profile, for $\ell\ne k$, since both are subsets of~$\vU$.)\looseness=-1
   \COMMENT{}

No trivial separation in~$\vS$ can ever distinguish two profiles of~$S$, because every consistent orientation of~$S$, and in particular every profile, contains it. Since trivial separations in~$\vS_\ell$ are also trivial in any~$\vS_k\supe \vS_\ell$, this implies that trivial separations in~$\vS_\ell$ cannot distinguish an $\ell$-profile from a $k$-profile in~$\vU$ for any $\ell\le k$. 

However, a separation $\vs$ distinguishing two profiles can be small; indeed there can be profiles of the same separation system~$\vS$ that differ only in how they orient such a separation~$s$.%
   \COMMENT{}
   For $\ell < k$, however, a small separation $\vs\in \vS_\ell$ can be nontrivial in~$\vS_\ell$ but trivial in~$\vS_k$.%
   \COMMENT{}
   Such a separation~$s$ can distinguish an $\ell$-profile (which orients it as~$\sv$) from a $k$-profile (which must orient it as~$\vs$) in a common universe~$\vU$.

\subsection{Tree sets distinguishing profiles of a separation system}\label{subsec:feasible}%
   \COMMENT{}

Let $(\vS, \leq\,,\!{}^*)$ be a separation system insided some universe~$\vU$ of separations,%
   \COMMENT{}
   and $\P$ a set of profiles of~$S$.

Distinct profiles of~$S$ are always distinguished by some $s\in S$,%
   \COMMENT{}
   so $S$ distinguishes~$\P$. Our aim is to find a nested (and, in particular, small)%
   \COMMENT{}
   subset $T$ of $S$ that still distinguishes all of~$\P$. Since trivial separations do not help in distinguishing profiles, $\vT$~will normally be a tree set, though not necessarily a regular one.%
   \COMMENT{}

Moreover, we shall often ask that $T$ should be {\em canonical\/} for~$S$: that the map $(\vS,\P)\mapsto \vT$ should commute with isomorphisms of separation systems. In particular, $\vT$~will then be invariant (globally) under the automorphisms of~$\vS$ if $\P$ is.\looseness=-1

Let us start with an example showing that our task of finding a canonical nested subset $T\sub S$ that still distinguishes all of~${\P}$ is nontrivial, indeed will not be achievable without further assumptions on~$\P$ (such as \eqref{eq:feasible} below):

\begin{Bsp} \label{exmpl:feasible}
   Suppose $S$ consists of two crossing separations $r$ and~$s$, and ${\P}$ consists of the four profiles $\{\vr,\vs \}$, $\{\vr,\sv\}$, $\{\rv ,\vs\}$, $\{\rv, \sv \}$ of $S$. Since only the singleton subsets of~$S$ are nested, but each of $r$ and~$s$ distinguishes only one pair of profiles, $S$~has no nested subset $T$ that distinguishes ${\P}$.
\end{Bsp}

The problem with this example is, very loosely, that $\P$ is too large and diverse compared with~$S$. Indeed we can mend it by increasing~$S$, without materially changing~$\P$, as follows.

\begin{Bsp}
Add to~$\vS$ the separations $\vt := \vr\lor\vs$ and $\vtdash := \vr\land\vs$ together with their inverses.%
   \COMMENT{}
   Now $\P$ is no longer formally a set of profiles of this larger separation system~$\vSdash$, but in substance it is: every $P\in\P$ extends uniquely to a profile of~$\vSdash$. (Indeed, by (P) and consistency, $\{\vr,\vs\}\in\P$ extends only to~$\{\vr,\vs,\vt,\vtdash\}$ while $\{\rv,\sv\}$ extends only to~$\{\rv,\sv,\tv,\tvdash\}$; for the other two elements of~$\P$ consistency forces us to add~$\tv$ and~$\vtdash$. But these extensions are indeed profiles of~$S'$.) And now there are two tree sets in~$S'$ that distinguish these four profiles, $\{r,t,t'\}$ and $\{s,t,t'\}$.

Neither of these tree sets is canonical as yet, since they map to each other under an automorphism of~$\vSdash$.%
   \COMMENT{}
   However if we add the other two corner separations too, our separation system will be rich enough: our four profiles still remain materially unchanged, but the four corner separations form a canonical tree set that distinguishes them all.
\end{Bsp}

This motivates the following condition, which we shall see will be sufficient for the existence of a canonical tree set $T\sub S$ that still distinguishes~$\P$. Let us say that $S$ \emph{scatters}~${\P}$ if the following holds:%
   \COMMENT{}
\begin{txteq} \label{eq:feasible}
   Whenever $\vr,\vs \in \vS$ cross and there are profiles $P, P^\prime \in {\P}$ with $\vr,\vs \in P$ and $\rv, \sv \in P^\prime$, there exists a separation $\vt \in P$ such that $\vr\lor\vs\le\vt$.%
   \COMMENT{}
\end{txteq}

\noindent
(By symmetry, there is then also be a separation $\tvdash\in P'$ such that $\rv\lor\sv\le\tvdash$.)

Note that the separation~$t$ in~\eqref{eq:feasible} distinguishes $P$ from $P^\prime$, so $t$~is $\P$-relevant. Indeed while $\vt\in P$, we have $\tv \leq \sv\in P'$, so $\tv\in P'$ by the consistency of~$P'$ unless $s=t$. But $s=t$ is impossible, since $r$ crosses~$s$ but is nested with~$t$.

We shall also say that a subset $R\sub S$ {\em scatters $\P$\/} if $R$ distinguishes~$\P$ and scatters the set $\P\restricts R$ of profiles of~$R$ that $\P$ induces. For example, if $S$ scatters $\P$ then so does the set $R$ of all $\P$-relevant separations in~$S$.%
   \COMMENT{}

Condition~\eqref{eq:feasible} will help us find our nested set $T\sub S$ that still distinguishes~$\P$, as follows. Suppose we are trying to pick $T$ from~$S$ inductively. At some point we might wish to select for~$T$ some separation that distinguishes two profiles $P,P'\in\P$, and have two separations $\vr,\vs\in\vS$ that both do this. If $\vr$ and~$\vs$ cross, we can put at most one of them in~$T$. Condition~\eqref{eq:feasible} will now save us having to choose: we can pick $\vt$ instead, as this also distinguishes $P$ from~$P'$. Moreover, since $t$ is nested with both $r$ and~$s$, adding it to~$T$ will allow us to keep our options open about $r$ or~$s$: if they are eligible for~$T$ now, they will still be eligible once we have added~$t$, so we can still settle for either one of them (or neither) later.

\goodbreak

An ideal starting point for finding a tree set $T \subseteq S$ of separations that still distinguishes ${\P}$
would be a separation $t$ that distinguishes some pair of profiles in ${\P}$ and is nested with {\em all\/} of~$S$: putting such a separation~$t$ in~$T$ will achieve something, and will do no harm to the desired nestedness of our eventual~$T$.%
   \COMMENT{}

Looking for such a separation~$t$ seems like a long shot. And indeed, condition~\eqref{eq:feasible} cannot ensure its existence: since \eqref{eq:feasible} speaks only about $\P$-relevant separations, the separation~$t$ it provides need not be nested with all of~$S$.%
   \COMMENT{}
   But since only the $\P$-relevant separations in $S$ matter for us, the following lemma will suffice:

\begin{Lem}\label{lem:extremal}
   If $S$ is ${\P}$-relevant and scatters~${\P}$, then every maximal $\vs\in\vS$ is nested with all the other separations in~$S$.
\end{Lem}

\begin{proof}
   Suppose some $r \in S$ crosses~$s$. Our aim is to find a pair $P,P'\in\P$ and an orientation $\vr$ of~$r$ such that $\vr,\vs\in P$ and $\rv,\sv\in P'$: then $\vs < \vt$ for the separation $\vt$ from~\eqref{eq:feasible}, contradicting the maximality of~$\vs$. (The inequality is strict, because $t$ is nested with~$r$ but $s$ is not.)

 As $s$ is ${\P}$-relevant, there are profiles $P,P^\prime \in {\P}$ such that $\vs\in P$ and $\sv\in P'$. Let $\vr$ be the orientation of~$r$ that lies in~$P$. If $\rv\in P'$ we are done, so assume that also $\vr\in P'$. Since $r$ is $\P$-relevant, there exists $P''\in\P$ such that $\rv\in P''$. If $P''$ contains~$\sv$, renaming $P''$ as~$P'$ yields the desired pair. If $P''$ contains~$\vs$, renaming $P''$ as~$P$ and swapping the names of $\vr$ and~$\rv$ yields the desired pair~$P,P'$.
\end{proof}

Every $\vs$ as in Lemma~\ref{lem:extremal} distinguishes a unique profile $P_s$ from the rest of~${\P}$:\looseness=-1

\begin{Lem}\label{lem:unique}
   If $S$ scatters~$\P$, and $\vs \in \vec{S}$ is maximal among the $\P$-relevant separations in~$\vS$,%
   \COMMENT{}
   then $\vs$ lies in a unique profile $P_{\vec s} \in {\P}$. Thus, $s$~distinguishes $P_{\vec s}$ from every other profile in~$\P$.
\end{Lem}

\begin{proof}
   As $s$ is ${\P}$-relevant, there are profiles $P,P' \in {\P}$ such that $\vs\in P$ and $\sv\in P'$.

To show that $P_{\vec s}:=P$ is unique, suppose there is another profile $P'' \in {\P}$ that  contains~$\vs$. As $P''\ne P$, there exists $r \in S$ with orientations $\vr \in P$ and $\rv \in P''$. Since $P$ and $P''$ orient $s$ identically but $r$ differently, we have $r \neq s$. Hence if $r$ is nested with~$s$, then $\vr < \vs$ or $\rv < \vs$ by the maximality of~$\vs$.%
   \COMMENT{}
   This contradicts the consistency of $P''$ or~$P$, respectively.

So $r$ crosses~$s$. If $\rv\in P'$, we apply \eqref{eq:feasible} to find $\vs  < \vt\in P$,%
   \COMMENT{}
   which contradicts the maximality of~$\vs$. If $\vr\in P'$, we apply~\eqref{eq:feasible} with $P''$ (as the $P$ in~\eqref{eq:feasible}) and~$P'$, swapping the names of $\vr$ and~$\rv$, to obtain $\vs  < \vt\in P''$ with a similar contradiction.%
   \COMMENT{}
\end{proof}

Consider an automorphism $\alpha\colon \vs\mapsto\vs{}^\alpha$ of the separation system $(\vS, \leq\,,\!{}^*)$. For subsets $P$ of~$\vS$ we write $P^\alpha \defgl \{\,\vs{}^\alpha \mid 
\vs \in P\,\}$, and put ${\P}^\alpha \defgl \{\,P^\alpha \mid P \in {\P}\,\}$ for sets of such subsets. Since $\alpha$ preserves (P) and consistency, $\P^\alpha$~is again a set of profiles of~$S$.

Note that $\alpha$ acts naturally also on the set~$S$ of unoriented separations, mapping $s=\{\vs,\sv\}$ to $s^\alpha \defgl \{\vs{}^\alpha, \sv{}^\alpha\}$. For sets $T\sub S$ we write $T^\alpha \defgl {\{\,s^\alpha \mid s \in T\,\}}$. If $R\sub S$ scatters~$\P$, then clearly $R^\alpha$ scatters~$\P^\alpha$.%
   \COMMENT{}

We are now ready to prove that separation systems scattering a set of profiles contain canonical tree sets that still distinguish these profiles. The main trick in the proof, which builds this tree set~$T$ inductively, is to toggle between adding separations to~$T$, which decreases the set~$\P$ of profiles not yet distinguished, and reducing $S$ to its $\P$-relevant subset, so that we can re-apply Lemmas~\ref{lem:extremal} and~\ref{lem:unique}.\looseness=-1

\begin{thm} \label{thm1} {\rm (Canonical tangle-tree theorem for separation systems)}\\
For every finite separation system $(\vS, \leq\,,\!{}^*)$ inside some universe of separations,%
   \COMMENT{}
   and every set $\P$ of profiles of~$S$ which $S$ scatters, there is a tree set $T\sub S$ of separations such that
   \begin{enumerate}[label = \rm{(\roman*)}]\itemsep0pt
      \item $T$ distinguishes ${\P}$;
      \item  $T$ is ${\P}$-relevant;
      \item $T$ is regular if all $P\in\P$ are regular, or if every separation in $S$ is $\P$-relevant%
   \COMMENT{}
   and no $P\in\P$ is special.
   \end{enumerate}
These sets $T=T(S,\P)$ can be chosen so that $T(S,\P)^\alpha = T(S',{\P}^\alpha)$%
   \COMMENT{}
   for every isomorphism $\alpha\colon (\vS, \leq\,,\!{}^*)\to (\vSdash, \leq\,,\!{}^*)$ of separation systems.%
   \COMMENT{}
   In particular, $T$~is invariant under any automorphism of~$(\vS, \leq\,,\!{}^*)$ that maps $\P$ to itself.%
   \COMMENT{}
\end{thm}

Theorem~\ref{thm1} will be the basic building block for a more general canonical tangle-tree theorem we shall prove in the next section. There, we shall find tree sets in universes of separations whose profiles they can distinguish even when these are not profiles of the same separation system.

\medbreak{\bf Proof of Theorem~\ref{thm1}.}
   We begin by proving the theorem without statement~(iii), applying induction on $|{\P}|$ with both $S$ and~$\P$ variable.
   For $|{\P}| \leq 1$ there is no pair of profiles to distinguish, so $T(S,\P) \defgl \emptyset $
   satisfies (i) and~(ii) trivially. For the last statement note that $|{\P}^\alpha| = |{\P}|$, so $ T(S,\P)^\alpha = \emptyset = T(S',{\P}^\alpha)$.

Assume now that $|{\P}| \geq 2$. Let $R \subseteq S$ be the set of ${\P}$-relevant separations in~$S$. As $|\P|\ge 2$ we have $R\ne\emptyset$.%
   \COMMENT{}
   As remarked after~\eqref{eq:feasible}, $R$~still scatters~$\P$.%
   \COMMENT{}
   By Lemma~\ref{lem:extremal} applied to~$\P\restricts R$, the set $R^+$ of maximal elements of~$\vR$ is nested with all of~$R$, and hence so is the set $T^+$%
   \COMMENT{}
   of all (unoriented) separations in~$R$ that have an orientation in~$R^+$. In particular, $T^+$~itself is nested.

By Lemma~\ref{lem:unique}, each $\vs \in R^+$ distinguishes a unique profile $P_{\vec s} \in {\P\restricts R}$ from the rest of ${\P\restricts R}$.%
   \COMMENT{}
   Let
 $${\P}^- \defgl {\P\restricts R} \sm \{\,P_{\vec s} \mid \vs \in R^+\}.$$
 As $|{\P}^- | < | {\P} |$,%
   \COMMENT{}
   our induction hypothesis%
   \COMMENT{}
   provides a nested set $T^-\sub R$ that satisfies (i) and~(ii) for~$\P^-$. Let 
 $$T \defgl T^+\cup T^-.$$
 Clearly, $T$~is nested. It is $\P$-relevant since $T\sub R$. Thus, $T$~is a tree set satisfying~(ii). By (i) for $\P^-$ and the choice of~$T^+$, it distinguishes $\P\restricts R$ and hence also~${\P}$, so $T$ also satisfies~(i). 
  
   The last statement of the theorem is straightforward to check.%
   \COMMENT{}
   It follows from the fact that if a separation $r\in S$ is $\P$-relevant then $r^\alpha$ is $\P^\alpha$-relevant, and that $\alpha$ maps the maximal elements of~$\vR$ to those of~$\vR^\alpha$. This completes our proof of the theorem without statement~(iii).

For a proof of~(iii), suppose $T$ is irregular. Then some $s\in T$ has a small orientation, $\sv$~say. By~(ii) we have $s\in R$, so there exists $P\in\P$ with $\vs\in P$. This $P$ is irregular.%
   \COMMENT{}
   To complete the proof of~(iii), we show that if $S=R$%
   \COMMENT{}
   then $P$ is special.\looseness=-1

Note that $\vs$ is maximal in~$\vR\,$: for any $\vr > \vs$ in~$\vS$ its inverse $\rv < \sv\le\vs$ is trivial, so $r\notin R$, because only nontrivial separations can distinguish profiles.%
   \COMMENT{}
   By our assumption that $S=R$ and Lemma~\ref{lem:extremal},%
   \COMMENT{}
   the separation $s$~is nested with every other separation in~$S$. By its maximality, $\vs$~points to no other separation in~$S$, so every $r\in S\sm\{s\}$ has an orientation $\vr < \vs\in P$. By its consistency, $P$~contains~$\vr$ rather than~$\rv$. Thus, $P = \{\,\vr\mid\vr\le\vs\,\}\sm\{\sv\}$, i.e., $P$~is special.%
   \COMMENT{}
   \qed\bigbreak

Let us add a few remarks. It may seem that assertion~(ii) would come for free once we have a set $T$ satisfying the other assertions, simply by chosing it inclusion-minimal subject to~(i). This, however, will lose us the last statement, canonicity. Indeed it can happen that $T$ cannot be chosen minimal:

\begin{Bsp}\label{nonmin}
Let $G$ be a graph obtained by extending a complete graph on $k$ vertices to three otherwise disjoint complete subgraphs $B_1, B_2, B_3$ of~$n>k$ vertices each. Each of these $B_i$ is an $n$-block that induces an $n$-profile on the set $S$ of the three $k$-separations $\{A_i,B_i\}$ of $G$ consisting of $(A_i, B_i)$ and $(B_j,A_j)$ for $j\ne i$. For the set $\P$ of these profiles, Theorem~\ref{thm1} finds $T=S$. This set is not minimal with assertion~(i), since any two of the three separations in~$T$ suffice to separate all three profiles. But only if $T$ contains all three of them will it be closed under the automorphisms of~$\vS$, which all map $\P$ to itself.%
   \COMMENT{}
\end{Bsp}

Note also that without any assumptions about $\P$, such as the premise in~(iii), there need not be a regular tree set~$T$ as in Theorem~\ref{thm1}. Here is a typical example:

\begin{Bsp}
Let $\vS$ itself be a tree set,%
   \COMMENT{}
   inside some universe~$\vU$ of separations, such that $\vr\lor\vs\in\vU\sm\vS$ for all incomparable $\vr,\vs\in\vS$.%
   \COMMENT{}
   Let $\vs$ be maximal in~$\vS$. Then every $r\in S$ has an orientation $\vr\le\vs$, so $P={\{\,\vr\mid\vr\le\vs\,\}}$ is a profile of~$S$. By replacing $\vs$ with $\sv$ in~$P$ we obtain another profile~$P'$ of~$S$: note that $P'$ is again consistent, since by the maximality of~$\vs$ in~$\vS$ its new separation $\sv$ cannot be part of an inconsistent pair of separations.%
   \COMMENT{}
   Then $s$ distinguishes $P$ from~$P'$, it is the only separation that does, and hence it will lie in the tree set~$T$ satifying Theorem~\ref{thm1}.

However, it is easy to construct instances of this where $\sv$ is small, making $P$ special. For example, $S$~might be the set of separations induced by a \td\ of a graph~$G=(V,E)$, with $\vs = (V,A)$ corresponding to an edge $uv$ at a leaf~$v$ of the decomposition tree.%
   \COMMENT{}
   In this case $T$~will not be regular, because it contains the improper separation~$s$.
\end{Bsp}

\subsection{Tree sets distinguishing profiles in separation universes} \label{subsec:robust}%
   \COMMENT{}

Let $\vU\, = (\vU,\le\,,\!{}^*,\vee,\wedge,|\ |)$ be a submodular universe%
   \COMMENT{}
  of separations. Our aim is to show that $\vU$ contains a canonical tree set $T$ that distinguishes all its%
   \COMMENT{}
   profiles.

Theorem~\ref{thm1} provides a start for fixed~$k$:

\begin{cor}\label{easycor}
If $\vU$ contains no separation of order less than~$k-1$, then there is a tree set $T\sub S_k$ that distinguishes all the $k$-profiles in~$\vU$ and is invariant under its automorphisms.
\end{cor}

\begin{proof}
This follows from Theorem~\ref{thm1} once we have checked that $S_k$ scatters all its profiles, the $k$-profiles in~$\vU$. Note first the separations in~$S_k$ have order exactly~$k-1$, by assumption. Given $\vr,\vs\in P$ as in~\eqref{eq:feasible}, we have $|\vr\land\vs|\ge k-1$ by assumption,%
   \COMMENT{}
   and hence $|\vr\lor\vs|\le k-1$ by submodularity, since $r$ and~$s$ have order exactly~$k-1$. Hence $\vt:=\vr\lor\vs\in\vS_k$, and therefore $\vt\in P$ by (P) and $\vr,\vs\in P$.
\end{proof}

In general we do not have the helpful assumption that $|s|\ge k-1$ for all $s\in U$. But our plan is to mimic it, by considering for $k=1,2,\dots$ in turn those profiles that can be distinguished by a separation of order~$<k$. At each step, we may then assume inductively that all the profiles we still need to distinguish cannot be distinguished by a separation of smaller order, just as in the premise of Corollary~\ref{easycor}.

The challenge with this approach is that, while each application of Theorem~\ref{thm1} provides a tree set of separations, we also have to ensure that these tree sets are nested with each other. We shall achieve this in the end, but it will need some care.

Let us say that two profiles in~$\vU$%
   \COMMENT{}
   are \emph{$k$-distin\-guish\-able} if some separation of order at most~$k$ distinguishes them. The smallest $k$ for which distinguishable profiles $P,P^\prime$ in~$\vU$ are $k$-distinguishable is denoted by $\kappa(P,P^\prime)$. If $s\in U$ distinguishes $P$ from~$P^\prime$ and has order $|s| = \kappa(P, P^\prime)$, we say that $s$, $\vs$ and~$\sv$ distinguish $P$ from $P^{\prime}$ \emph{efficiently\/}. A~set $T\sub U$ distinguishes a set $\P$ of profiles in~$\vU$ {\em efficiently\/} if%
   \COMMENT{}
   every two profiles in~$\P$ are distinguished efficiently by some separation in~$T$.

Consider a pair of crossing separations $\vr,\vs\in\vU$ and their corner separations $\vr_1:=\rv\land\vs$ and $\vr_2:= \rv\land\sv$. Then $\vr_1,\vr_2\le\rv$. If $\rv$ is the supremum in~$\vU\,$%
   \COMMENT{}
   of $\vr_1$ and~$\vr_2$, i.e., if $\rv = \vr_1\lor\vr_2$,%
   \footnote{\label{robustfootnote}If $U$ is the set of all bipartitions of a set, then this is always the case. Profiles in matroids, therefore, or clusters in data sets, will always be robust in the sense defined below.}
   then $\vr_1, \vr_2$ and~$\vr$ cannot lie in a common profile, because this would violate~(P). But in general $\vr_1, \vr_2$ and~$\vr$ might lie in a common profile, say in the $k$-profile~$P$. This can only happen if $|s|\ge k$, since otherwise $P$ would contain $\sv$ or~$\vs$, but both $\{\sv,\vr,\vr_1\}\sub P$ and $\{\vs,\vr,\vr_2\}\sub P$ would violate~(P). If it does happen then, intuitively, the separation $s\in U\sm S_k$ looks a bit as though it splits~$P$ in half. This%
   \COMMENT{}
   can cause problems, so let us give a name to certain%
   \COMMENT{}
   profiles that are not split in this way.

Let us say that a profile $P$ in~$\vU$ is {\em robust\/} if it is $n$-robust for every~$n$, and $n$\emph{-robust}%
   \COMMENT{}
   if for every $\vr \in P$ and every $s\in S_n$ the following holds:
$$\text{If $\rv\land\vs$ and $\rv\land\sv$ both have order $<|r|$, they do not both lie in~$P$\rlap.}\eqno\rm(R)$$

Every $n$-profile is $n$-robust, because it has to contain $\sv$ or~$\vs$ and so, by~(P), cannot also contain both $\vr$ and~$\rv\land\vs$ or both $\vr$ and~$\rv\land\sv$. Clearly, $n$-robust $k$-profiles induce $n$-robust $\ell$-profiles for all~$\ell<k$, and every $n$-robust profile is also $m$-robust for every $m < n$. 

\begin{Lem}\label{lem:nested2}
   Let $n$ be a positive integer. Let $r\in U$ be a separation that efficiently distinguishes two $n$-robust profiles $P,P'$ in~$\vU$, and let $s\in U$ be a separation that efficiently distinguishes two profiles $\hat P, \hat P'$ in~$\vU$. If $|r| < |s| < n$, then $r$ has%
   \COMMENT{}
   an orientation~$\vr$ such that either $\vr \land \vs$ or $\vr \land \sv$ efficiently distinguishes $\hat{P}$ from~$\hat{P}^\prime$.
\end{Lem}

\begin{proof}
As $|r|<|s|$ and $s$ distinguishes $\hat P$ from $\hat P'$, efficiently, $\hat P$~and $\hat P'$ also orient~$r$,%
   \COMMENT{}
   and they do so identically,%
   \COMMENT{}
   say as~$\rv$.%
   \COMMENT{}
   Let $\vs$ be the orientation of $s$ in~$\hat{P}$; then $\sv \in \hat{P}^\prime$.
If $|\vr \land \vs| \le |s|$ then $\vr \land \vs$ distinguishes $\hat P$ from~$\hat P'$ (and efficiently):%
   \COMMENT{}
   it lies in~$\hat P$ by $\vr\land\vs \le \vs\in\hat P$ and the consistency of~$\hat P$, but not in~$\hat P'$ by $\rv,\sv\in\hat P'$ and~(P). Similarly, if $|\vr\land\sv| \le |s|$ then $\vr\land\sv$ efficiently distinghuishes $\hat P$ from~$\hat P'$.%
   \COMMENT{}
   So let us assume that $|\vr\land\vs|, |\vr\land\sv| > |s|$, and derive a contradiction.

By submodularity, $\vr_1 = \rv\land\vs$ and $\vr_2 = \rv\land\sv$ have order $ < |r|$. As $r$ distinguishes $P$ from~$P'$, it has orientations $\vr\in P$ and $\rv\in P'$. As it does so efficiently, $P$~and~$P'$ orient $r_1$ and~$r_2$ identically.%
   \COMMENT{}
   As $P'$ contains $\vr_1,\vr_2\le\rv$ by consistency, we thus have  $\vr_1,\vr_2\in P$ too, as well as $\vr\in P$ by assumption. This contradicts the $n$-robustness of~$P$, since $s\in S_n$.\looseness=-1
\end{proof}

The reader will have noticed that, in the proof of Lemma~\ref{lem:nested2}, we did not in fact use that $P$ and $P'$ themselves are $n$-robust:%
   \COMMENT{}
   we only used the implied $n$-robustness of the $k$-profiles they induce for $k = |r|+1$.
Let us say that two profiles $P,P^\prime$ in~$\vU$ are \emph{$n$-robustly distinguishable} if they are distinguishable and the distinct $k$-profiles they induce for $k =\kappa(P,P^\prime) + 1$ are both $n$-robust.%
   \COMMENT{}

A set $\P$ of profiles in~$\vU$ is {\em $n$-robust\/} if any distinct $P,P'\in\P$ are $n$-robustly distinguishable%
   \COMMENT{}
  and satisfy $\kappa(P,P')<n$. If a set $\P$ of distinguishable%
   \COMMENT{}
   profiles is $n$-robust for
 $$n = \max\{\,\kappa(P,P')\mid P,P'\in\P\,\}+1,$$
   we call it {\em robust\/}. This will be the case if every $P\in\P$ is robust and the profiles in~$\P$ are pairwise distinguishable.%
   \COMMENT{}
   But there can be robust sets of profiles that are not individually robust, as their individual robustness may fail just for $n$ much%
   \COMMENT{}
   larger than the order of separations needed to distinguish~$\P$.%
   \COMMENT{}

\begin{thm}\label{thm2} {\rm (Canonical tangle-tree theorem for separation universes)}\\
   Let $\vU = (\vU,\le\,,\!{}^*,\vee,\wedge,|\ |)$ be a submodular universe%
   \COMMENT{}
 of separations. Then for every%
   \COMMENT{}
   robust set $\P$ of profiles in~$\vU$ there is a nested set $T=T({\P}) \subseteq U\!$ of separations such that:
   \begin{enumerate}[label = \rm(\roman*)]\itemsep0pt
      \item every two profiles in~${\P}$ are efficiently distinguished by some separation~in~$T$;
      \item every separation in~$T$ efficiently distinguishes a pair of profiles in ${\P}$;
      \item for every automorphism $\alpha$ of $\vU$%
   \COMMENT{}
    we have $T({\P}^{\alpha}) = T({\P})^{\alpha}$;
      \item if all the profiles in $\P$ are regular, then $T$ is a regular tree set.%
   \COMMENT{}
   \end{enumerate}
\end{thm}

\begin{proof}
   For every $\P$ and $k>0$ let ${\P}_k$ denote the set of all $n$-robust%
   \COMMENT{}
   $k $-profiles induced by profiles in~${\P}$, where $n = \max\{\,\kappa(P,P')\mid P,P'\in\P\,\}+1$. As $\P$ is robust, distinct $P,P'\in\P$ induce distinct elements of~$\P_k$ for $k = \kappa(P,P')+1$. Hence any set $T\sub U$ distinguishing every $\P_k$ will also distinguish~$\P$.

We shall construct, simultaneously for all~$\P$,%
   \COMMENT{}
   nested sets $T_0\sub T_1\sub\dots$ of separations $T_k = T_k(\P)$ such that, for all $k\ge 1$:
   \begin{enumerate}[label=(\Roman*)]\itemsep0pt
      \item $T_k$ efficiently distinguishes~${\P}_k\,$;
      \item every $s \in  T_k \sm T_{k -1}$ efficiently distinguishes some pair of profiles in ${\P}_k$;
      \item for every automorphism $\alpha$ of $\vU$ we have $T_k({\P}^{\alpha}) = T_k({\P})^{\alpha}$.
   \end{enumerate}
Then, clearly, $T:=\bigcup_{k=1,2,\dots} T_k$ will satisfy (i)--(iii) and be nested,%
   \COMMENT{}
   as desired. Statement~(iv) will be an immediate consequence of~(ii), because small separations cannot distinguish regular profiles.%
   \COMMENT{}
  
We start our construction with $T_{0} = \emptyset$.%
   \COMMENT{}
   Now consider $k >0$, and assume that for all $0 < \ell < k$ and all~$\P$ we have constructed $T_{\ell} = T_{\ell}(\P)$ so as to satisfy (I)--(III).

If $k=1$, put $Q:=\emptyset$ and $\P_Q:=\P_1$, and let $S_Q$ be the set of all $\P_1$-relevant separations in~$S_1$. Clearly, $S_Q$~distinguishes $\P_Q$ efficiently.

If $k>1$, suppose some distinct $P,P^{\prime}\in {\P}_k$ are not yet 
distinguished by $T_{k -1}$. The $(k-1)$-profiles that $P$~and~$P^{\prime}$ induce are again $n$-robust, so they lie in~$\P_{k-1}$, and hence by~(I) cannot be distinct.%
   \COMMENT{}
   We denote the set of all profiles in 
${\P}_k$ that induce the same profile $Q \in {\P}_{k -1}$ by ${\P}_Q$.
As ${\P}_k$ is the union of all such~${\P}_Q$, it will suffice to find for every $Q \in {\P}_{k -1}$ a nested set $T_Q\sub U$ that
efficiently distinguishes~${\P}_Q$, and make sure that these 
$T_Q$ are are also nested with $T_{k -1}$ and with each other.

If $k>1$, consider any fixed $Q \in {\P}_{k -1}$. Let $S_Q$ be the set of all those separations $s\in S_k$ which distinguish some $\hat P,\hat P'\in\P_Q$ that are not distinguished by any separation in~$S_k$ that crosses fewer separations in~$T_{k -1}$ than $s$ does. Clearly, every separation in~$S_Q$ is $\P_Q$-relevant, and $S_Q$ distinguishes~${\P}_Q$ efficiently.%
   \COMMENT{}

Let us show that $S_Q$ is nested with~$T_{k -1}$.%
   \COMMENT{}
   Suppose there is a separation $s \in S_Q$ that crosses a separation $r \in T_{k -1}$. Then $k>1$, as $T_0 = \emptyset$. By the definition of~$S_Q$, there exist $\hat P,\hat P'\in\P_Q$ that are distinguished by~$s$ but by no separation in~$S_k$ that crosses fewer separations in~$T_{k -1}$ than $s$ does. By Lemma~\ref{lem:nested2},%
   \COMMENT{}
   $r$~and~$s$ have orientations $\vr$ and~$\vs$ such that $\vr \land \vs$ or $\vr \land \sv$ efficiently distinguishes $\hat P$ from~$\hat P'$. These corner separations of $r$ and~$s$ are not only, unlike~$s$, nested with~$r$ but also, by Lemma~\ref{lem:fish}, with every separation $t\in T_{k -1}$ that is nested with~$s$, because $t$ is also nested with~$r$  since both lie in~$T_{k -1}$. This contradiction to the choice of $\hat P$ and~$\hat P'$%
   \COMMENT{}
   completes our proof that $S_Q$ is nested with~$T_{k -1}$.

Our next aim is to find a nested subset $T_Q$ of $S_Q$ that still distinguishes all of~${\P}_Q$. As $S_Q$ distinguishes~${\P}_Q$, it suffices to find such a set $T_Q$ that distinguishes $\P_Q\restricts S_Q$.%
   \COMMENT{}
   This will exist by Theorem~\ref{thm1} if $S_Q$ scatters~$\P_Q\restricts S_Q$.

To prove this, consider two profiles $P,P^{\prime} \in \P_Q\restricts S_Q$ and two crossing separations $r,s \in S_Q$ 
with orientations $\vr,\vs \in P$ and $\rv, \sv \in P^{\prime}$. Let us show that both $\vr\land\vs$ and $\vr\lor\vs$, like $r$ and~$s$, have order~$k-1$. Neither of them can have smaller order than~$k-1$: if one of them did then it would, by (P) and consistency,%
   \COMMENT{}
   distinguish the profiles in $\P_Q$ that induce $P$ on~$S_Q$ from those that induce~$P'$, so two such profiles in~$\P_Q$%
   \COMMENT{}
   would induce distinct profiles in~$\P_{k-1}$, whereas in fact they all induce~$Q$. Thus, $|\vr\land\vs|, |\vr\lor\vs|\ge k-1$. Since $|\vr\land\vs| + |\vr\lor\vs| \le |r|+|s|$ by submodularity, and $|r|=|s|=k-1$ as $r,s\in S_Q$, we thus have $|\vr\land\vs| = |\vr\lor\vs| = k-1$ as claimed.

Thus, $\vr\lor\vs$ lies in~$\vS_k$, and it distinguishes $P$ from~$P'$. By Lemma~\ref{lem:fish} it is nested with every $t\in T_{k-1}$, because $r,s\in S_Q$ are nested with~$T_{k-1}$ (as shown earlier).%
   \COMMENT{}
   By definition of~$S_Q$, this means that $\vr\lor\vs$ lies in~$\vS_Q$. But then it also lies in~$P$, by $\vr,\vs\in P$ and~(P). This completes our proof that $S_Q$ scatters~${\P}_Q$.

Applying Theorem~\ref{thm1} to $\P_Q\restricts S_Q$ in $(\vS_Q, \leq\,,\!{}^*)$ for each~$Q\in {\P}_{k -1}$ (respectively, for $Q=\emptyset$ if $k=1$),
we obtain a family of sets $T_Q\sub S_Q$ of separations distinguishing~${\P}_Q$.%
   \COMMENT{}
    Each of these~$T_Q$ is nested, and nested with~$T_{k -1}$. In order to show that they are also nested with each other, consider any $s\in T_Q$ and $s'\in T_{Q^\prime}$ for distinct $Q,Q^\prime \in \P_{k -1}$. By (I) for $k-1$ there is a separation $r \in T_{k -1}$ with orientations $\vr \in Q$ and $\rv \in Q^\prime$. Since $T_Q$ is nested with~$T_{k -1}$, the separations $r$~and~$s$ have comparable orientations. So either $\vr$ or $\rv$ points towards~$s$; let us show that $\vr$ does.%
   \COMMENT{}

Since $s\in S_Q$ it is $\P_Q$-relevant,%
   \COMMENT{}
   so it has orientations $\vs \in P_1$ and $\sv \in P_2$ for some $P_1,P_2 \in {\P}_Q$. 
As $P_1$ and $P_2$ induce~$Q$, they both contain~$\vr$. Hence if $\rv$ points towards~$s$, then either $\rv < \vs$ or $\rv < \sv$,%
   \COMMENT{}
  contradicting the consistency of $P_1$ or~$P_2$, respectively. So $\vr$ points towards~$s$, as claimed.

Similarly, $\rv$~points towards~$s'$. We thus have $\vsdash < \vr < \vs$ for suitable orientations of $s$ and~$s'$. In particular, $s$ and~$s'$ are nested.

We have shown that the $T_Q$ are nested sets of separations that are also nested with each other and with~$T_{k-1}$. Hence, $T_k:= T_Q$ for $k=1$, and otherwise
   $$T_k \defgl T_{k -1} \cup\! \bigcup_{Q \in {\P}_{k -1}}\!T_Q$$
   \vskip-3pt\noindent\COMMENT{}%
is nested too, and by construction%
   \COMMENT{}
   it satisfies (I) and~(II).

In order to verify~(III), let $\alpha$ be an arbitrary automorphism of $(\vU,\le\,,\!{}^*,\vee,\wedge,|\ |)$. If $k=1$, then (III) holds by definition of~$T_1$, which we obtained from Theorem~\ref{thm1} as a canonical tree set $T\sub S_1$ distinguishing~$\P_1$. Now consider $k>1$. 
Note that $\P_\ell^\alpha = (\P_\ell)^\alpha$ for all~$\ell$,%
   \COMMENT{}
   by definition of~$\P_\ell$. Hence $S_{Q^\alpha} = (S_Q)^\alpha$ for every $Q\in\P_{k-1}$, by definition of~$S_Q$ for $\P$ and $S_{Q^\alpha}$ for~$\P^\alpha$, and (III) for~${k-1}$.%
   \COMMENT{}
   Therefore $\P^\alpha_{Q^\alpha}\restricts S_{Q^\alpha} = (\P_Q\restricts S_Q)^\alpha$, so Theorem~\ref{thm1} implies that $T(\P^\alpha_{Q^\alpha}\restricts S_{Q^\alpha}) = T(\P_Q\restricts S_Q)^\alpha$, for all $Q \in {\P}_{k -1}$. Hence, again using (III) for~$k-1$,\looseness=-1
   \begin{align*}
      T_k({\P}^\alpha) &= T_{k -1}({\P}^\alpha) \cup 
      \bigcup_{Q^\alpha \in {\P}^\alpha_{k -1}} T(\P^\alpha_{Q^\alpha}\restricts S_{Q^\alpha})\\
      & = T_{k -1}({\P})^\alpha \cup \bigcup_{Q \in {\P}_{k -1}} T(\P_Q\restricts S_Q)^\alpha\\
      &= T_k({\P})^\alpha,
   \end{align*}
completing the proof of (I)--(III) for~$k$.
\end{proof}

\goodbreak

Let us add a few remarks. First, we stated~(iii) for automorphisms rather than arbitrary isomorphism of universes just to keep things simple; it clearly generalizes as in Theorem~\ref{thm1} (whose more general form we needed in the proof of Theorem~\ref{thm2}).%
   \COMMENT{}

Since trivial separations cannot distinguish profiles, one might expect that, by~(ii), $T$~will always be a tree set. However this is not the case. For example, $\vT$~might contain a small separation~$\vr$ that distinguishes two $k$-profiles in~$\P$. This $\vr$ will not be trivial in~$S_k$. But it may be trivial in some~$S_n$ with $n>k$. Then no $n$-profile in~$\vU$ will contain~$\rv$, but $r$ may still be needed for~$T_k$.
Under mild additional assumptions,%
   \footnote{For example, it suffices to assume that, when $\vs$ is small, $\vr\le\vs$ implies $|r|\le|s|$. This holds for separations of graphs and matroids.}%
   \COMMENT{}
   however, one can replace the premise in~(iv) with the weaker premise that no $P\in\P$ is special.

In applications, the robustness of $\P$ will often be a consequence of the fact that the profiles it contains are themselves robust, and often they will also be regular. Here is a leaner version of Theorem~\ref{thm2} for robust regular profiles:

\begin{cor}\label{cor2}
   Let $\vU = (\vU,\le\,,\!{}^*,\vee,\wedge,|\ |)$ be a submodular universe of separations. For every%
   \COMMENT{}
   set $\P$ of pairwise distinguishable%
   \COMMENT{}
   robust regular profiles in~$\vU$ there is a regular tree set $T=T({\P}) \subseteq \vU$ of separations such that:\vskip-\medskipamount\vskip0pt
   \begin{enumerate}[label = \rm(\roman*)]\itemsep0pt
      \item every two profiles in~${\P}$ are efficiently distinguished by some separation~in~$T$;
      \item every separation in~$T$ efficiently distinguishes a pair of profiles in ${\P}$;
      \item for every automorphism $\alpha$ of $\vU$%
   \COMMENT{}
    we have $T({\P}^{\alpha}) = T({\P})^{\alpha}$.\qed
   \end{enumerate}
\end{cor}

\noindent
But note that we cannot simply omit all robustness requirements in Theorem~\ref{thm2} and Corollary~\ref{cor2}; see~\cite{confing} for a counterexample.

As a common basis for the proof of our Theorems~\ref{generalk} and~\ref{matroids}, let us note a further simplification of Theorem~\ref{thm2}. Rather than starting with any fixed set~$\P$ of distinguishable profiles, we find a tree set that distinguishes them all:%
   \COMMENT{}

\begin{cor}\label{cor3}
Let $\vU = (\vU,\le\,,\!{}^*,\vee,\wedge,|\ |)$ be a submodular universe of separations. There is a canonical%
   \COMMENT{}
   regular tree set $T\sub \vU$ that efficiently distinguishes all the distinguishable robust and regular profiles in~$\vU$.
\end{cor}

\begin{proof}
Let $\P'$ be the set of robust and regular profiles in~$\vU$, and let $\P$ be the set of its inclusion-maximal elements. Note that the profiles in~$\P$ are pairwise distinguishable.%
   \COMMENT{}
   By Corollary~\ref{cor2}, there is a canonical tree set $T\sub\vU$ that efficiently distinguishes~$\P$. We have to show that it efficiently distinguishes every pair $P_1,P_2$ of distinguishable profiles in~$\P'$.

Each $P_i$ is a $k_i$-profile for some~$k_i$. By definition of~$\P$ it extends to some~${P'_i\in\P}$, which is a $k'_i$-profiles for some $k'_i\ge k_i$. (Recall that all the profiles we are considering are profiles `in'~$\vU$, and hence $k$-profiles for some~$k$.) Since $T$ distinguishes the $P'_i$ efficiently, and any separation distinguishing the $P_i$ will also distinguish their extensions~$P'_i$, our assumption that the $P_i$ are distinguishable means that $T$ distinguishes the $P'_i$ by a separation of order at most~$\kappa(P_1,P_2)$.%
   \COMMENT{}
   This separation, therefore, has orientations in both $P_1$ and~$P_2$, and thus distinguishes these efficiently too.
\end{proof}

\subsection{Minimal distinguishing tree sets, non-canonical}\label{subsec:min}

As shown by Example~\ref{nonmin} we cannot, without losing canonicity, strengthen (ii) in Theorem~\ref{thm2} to saying that $T$ is minimal given~(i). But if we are prepared to make do without canonicity, we can strengthen (ii) even more: we can ensure that $T$ is minimal not only with~(i) but more generally with the property that it distinguishes~$\P$, efficiently or not. (In fact, we shall prove that every $T$ that is minimal with~(i) is minimal even as a $\P$-distinguishing set of separations; see Lemma~\ref{fun} below.)

Robertson and Seymour~\cite{GMX}%
   \COMMENT{}
   proved that every graph has a (non-canonical) \td\ that distinguishes any given set $\Theta$ of distinguishable tangles, in the sense that any distinct $\theta,\theta'\in\Theta$ disagree on some separation associated with the \td. Moreover, this decomposition has only $|\Theta|$~parts. It follows%
   \COMMENT{}
   that every tangle $\theta\in\Theta$ `lives in' one of these parts~-- the unique part towards which $\theta$ orients all the separations associated with the \td\ that it orients~-- and every part is home to a tangle from~$\Theta$.%
   \COMMENT{}

Although in our abstract separation system we have no vertices, and our tree set $T$ has no `parts', we can generalize this to a non-canonical version of our theorem by representing those parts as consistent orientations of~$T$. For every $P\in\P$, the subset $P\cap\vT$ of~$\vT$ is consistent, and hence extends to a consistent orientation of all of~$T$.%
   \footnote{We need here that  no element of $P\cap\vT$ is co-trivial in~$T$, which we shall ensure by requiring $\P$ to be regular. See~\cite[Lemma 4.1]{AbstractSepSys} for a formal proof that such extensions exist.}
   This extension is not in general unique, but for our minimal $T$ it will be.%
   \COMMENT{}
   Moreover, every consistent orientation of~$T$ can be obtained from some $P\in\P$ in this way. In this sense, every profile in~$\P$ will `live in' a unique consistent orientation of~$T$, and every consistent orientation of $T$ will be home to some $P\in \P$.

\begin{thm}\label{thm2min} {\rm (Non-canonical tangle-tree theorem for separation universes)}\\
   Let $\vU = (\vU,\le\,,\!{}^*,\vee,\wedge,|\ |)$ be a submodular universe of separations. For every robust set $\P$ of regular profiles in~$\vU$ there is a regular%
   \COMMENT{}
   tree set $T'\subseteq \vU$ of separations such that:
   \begin{enumerate}[label = \rm(\roman*)]\itemsep0pt
      \item every two profiles in~${\P}\!$ are efficiently distinguished by some separation~in~$T'$;
      \item no proper subset of~$T'$ distinguishes~$\P$;
      \item for every $P\in\P$ the set $P\cap\vTdash$ extends to a unique consistent orientation\penalty-200\ $\PT$ of~$T'$. This map $P\mapsto \PT$ from $\P$ to the set $\O$ of consistent orientations of~$T'$ is bijective.
   \end{enumerate}
\end{thm}

\noindent
Note that, by~(i), assertion (ii)~implies that%
   \COMMENT{}
    every separation in~$T'$ distinguishes two profiles in ${\P}$ efficiently, as earlier in Theorem~\ref{thm2}.

Also, the implicit assumption in the theorem that the profiles in $\P$ must be pairwise distinguishable (because we require $\P$ to be robust) is not a real restriction. Indeed, the tree set~$T'$ returned by Theorem~\ref{thm2min} will distinguish {\em all\/} the distinguishable robust and regular profiles in~$\vU$: just take as~$\P$ the set of all the maximal such profiles, and argue as in the proof of Corollary~\ref{cor3}.

\medbreak

For the proof of Theorem~\ref{thm2min} we need a fun little lemma about separating edge sets in trees, which may also be of interest in its own right. The intended application is that choosing $T'$ in Theorem~\ref{thm2} minimal with respect to~(i), ignoring canonicity, will automatically make it satisfy (ii) of Theorem~\ref{thm2min}. Recall that a {\em $U\!$-path\/} in a graph $G$ with $U\sub V(G)$ is a path that meets~$U$ in exactly its distinct ends.

\begin{Lem}\label{fun}
Let $G$ be a tree with an edge labelling $\ell\colon E(G)\to\mathbb N$. Let $U\sub V(G)$ and $F\sub E(G)$ be such that every $U\!$-path $P\sub G$ contains an edge $e\in F$ such that $\ell(e) = \min\{\,\ell(e')\mid e'\in P\,\}$ and $F$ is minimal with this property, given~$U$. Then for every $e\in F$ there is a  $U\!$-path in~$G$ whose only edge in~$F$ is~$e$.
\end{Lem}

\begin{proof}
Let us call an edge $e\in G$ {\em essential\/} on a path $P\sub G$ if $e$ is the unique edge of~$P$ in~$F$ with $\ell(e) = \min\{\,\ell(e')\mid e'\in P\,\}$. The minimality of~$F$ assumed in the lemma means that every edge in $F$ is essential on some $U\!$-path. If the lemma fails, there is an edge $e\in F$ that is not the only $F$-edge on any $U\!$-path. Let such an edge $e$ be chosen with $\ell(e)$ maximum, and let $P = uGv$ be a $U\!$-path on which $e$ is essential.

   \begin{figure}[htpb]
\centering
        \includegraphics{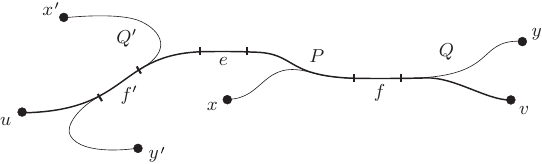}
        \caption{Paths in the proof of Lemma~\ref{fun}}
   \end{figure}

By definition, $e$~is not the only edge of~$P$ in~$F$; choose another, $f$~say, as close to $e$ as possible. Assume that $f$ lies on the segment $ePv$ of~$P$ strictly%
   \COMMENT{}
   between $e$ and~$v$. Then $\ell(f) > \ell(e)$ since $e$ is essential on~$P$, so by the choice of~$e$ there exists a $U\!$-path $Q=xGy$ whose only edge in~$F$ is~$f$. Let $x$ be the end of~$Q$ for which $f\notin eGx$. By the choice of~$f$,%
   \COMMENT{}
   this path $eGx\sub P\cup Q$ has no edge in~$F$.

If $uPe$ has no edge in~$F$, then $e$ is the unique $F$-edge on the $U\!$-path $uGx$, contradicting the choice of~$e$. So $uPe$ contains an edge $f'\in F$. Choose $f'$ as close to~$e$ as possible, and repeat the earlier argument with $f'$ in the place of~$f$ to find a $U\!$-path $Q'=x'Gy'$ whose only edge in~$F$ is~$f'$, with $f'\notin eGx'$ say.%
   \COMMENT{}
   By the choice of~$f'$, this path $eGx'\sub P\cup Q'$ has no edge in~$F$. Hence $e$ is the only $F$-edge on the $U\!$-path $xGx'$, contradicting the choice of~$e$.
\end{proof}

\noindent{\bf Proof of Theorem~\ref{thm2min}.}
Let $T$ be the regular tree set provided for~$\P$ by Theorem~\ref{thm2}. Then $T'=T$ satisfies (i) of our theorem. Now choose $T'\sub T$ minimal with this property.
\begin{txteq}\label{PT}
 For every $P\in\P$, the set $P\cap\vTdash$ extends uniquely to a consistent orientation~$\PT$ of~$T'$.
\end{txteq}
 
To prove this, note first that $P\cap\vTdash$ is consistent, because~$P$ is. By~\cite[Lemma~4.1]{AbstractSepSys} and our assumption that all $P\in\P$ are regular, $P\cap\vTdash$ extends to a consistent orientation $\PT$ of all of~$T'$.

For a proof that $\PT$ is unique, let $\ell := \max\{\, |r| : \vr\in P\cap\vTdash\}$. Since $P$ is an $m$-profile, for some~$m>\ell$, it orients {\em all\/} the separations of order~$\le\ell$ in~$T$. If $P\cap\vTdash$ extends to distinct consistent orientations of~$T'$ there is some $s\in T'$ which $P$ does not orient, which thus has order $k := |s| > \ell$. In fact, $s$~is not oriented by the down-closure $\dcl(P\cap\vTdash)$ of $P\cap\vTdash$ in~$\vTdash$,%
   \COMMENT{}
   because this is contained in all consistent orientations of~$T'$ extending~$P\cap\vTdash$, by definition of consistency.%
   \COMMENT{}

When $s$ was added to~$T$ in its construction, the reason was that $s$ distinguishes some profiles $P',P''\in\P$ inducing the same profile $Q$ of~$S_k$, with $\vs\in P'$ and $\sv\in P''$ say.%
   \COMMENT{}
   Notice that every $\vr\in P\cap\vTdash$ points to~$s$: as $T'$ is nested, $\vr$ is either smaller or greater than some orientation of~$s$, but $s$ has no orientation in $\dcl(P\cap\vTdash)\supe\dcl(\vr)$.%
   \COMMENT{}

Hence for every $\vr\in P\cap\vTdash$, which lies in $\vT_k$ by definition of $\ell$ and~$k$, we have either $\vr < \vs$ and thus $\vr\in P'\cap\vT_k\cap \vTdash = Q\cap\vTdash$ by the consistency of~$P'$ and $\vs\in P'$, or else $\vr < \sv$ and $\vr\in P''\cap\vT_k\cap \vTdash = Q\cap\vTdash$ by the consistency of~$P''$ and $\sv\in P''$. So on the separations in~$T'$ that it orients, $P$~agrees with~$Q$, and hence with both $P'$ and~$P''$.%
   \COMMENT{}
   Thus, $T'$~does not distinguish $P$ from either $P'$ or~$P''$. But $T'$ does distinguish all distinct profiles in~$\P$. So $P$ is not distinct from either $P'$ or~$P''$, i.e., $P'=P=P''$ contradicting the choice of $P'$ and~$P''$.%
   \COMMENT{}

This completes our proof of~\eqref{PT}. Thus, $P\mapsto\PT$ is a well-defined map from $\P$ to the set $\O$ of consistent orientations of~$T'$. Let us use this to prove assertion (ii) of our theorem.\looseness=-1

As shown in~\cite{TreeSets}, there is a tree~$G$ with edge set~$T'$ whose natural ordering on~$\vT'$, as defined in Section~\ref{subsec:nested}, coincides with our given ordering on~$\vTdash$. As with every finite tree, the consistent orientations of its edge set with respect to this natural ordering correspond to its nodes: an orientation of $E(G)$ is consistent if and only if it orients all of $E(G)$ towards some fixed node of~$G$~\cite{CDHH13CanonicalParts, TreeSets}. Since the consistent orientations of $E(G)$ are those of~$T'$,%
   \COMMENT{}
   we may thus think of $\O$ as the node set of~$G$.

Let us apply Lemma~\ref{fun} to $G=(\O,T')$ with $F=T'$ and $U = {\{\,\PT\mid P\in\P\}}$ and $\ell(s):= |s|$ for all $s\in T'$.%
   \COMMENT{}
   Note that a separation $s\in T'$ distinguishes two profiles $P,P'\in\P$ if and only if $s$ lies on the unique $\PT$--$\PT'$ path in~$G$: it is then, and only then, that orienting it towards the node~$\PT$ differs from orienting it towards the node~$\PT'$, and as mentioned earlier, an edge $\vs\in\vTdash$ is oriented towards a node $\PT$ of~$G$ if and only if $\vs$ lies in~$\PT$.%
   \COMMENT{}

Let us show that $F=T'$ satisfies the premise of the lemma. By definition of~$T'$, every two profiles $P,P'\in\P$ are distinguished by some $s\in T'$ with $\ell(s) = \kappa(P,P')$, but as soon as we delete an element $s$ of~$T'$ it loses this property.%
   \COMMENT{}
   Then there are $P,P'\in\P$ such that $s$ lies on the $\PT$--$\PT'$ path in~$G$ but every other edge $r$ on that path has a label $\ell(r) = |r| > \kappa(P,P') = |s| = \ell(s)$. After deleting $s$ from~$F$ no edge of mimimum label on that path will lie in~$F$. Hence, $F$ is as required in the lemma.

The lemma asserts that, therefore, each $s\in T'$ is the only edge in $T'$ that distinguishes some two profiles in~$\P$ depending on~$s$. This is statement~(ii) of our theorem. Moreover, it implies that $U = V(G)$: if $s$ is the only edge in $F=T'=E(G)$ on some $U$-path, then this path has length~1, and hence the ends of~$s$ lie in~$U$. Thus, our map $P\to\PT$ from $\P$ to~$\O$ is surjective. Since it is also injective, by~(i), our proof is complete.\qed%
   \COMMENT{}

\subsection{Symmetric profiles}\label{subsec:awful}

This section is for those readers that have become interested in abstract profiles and wonder whether the concept might still be improved. Readers primarily interested in their applications~-- to graphs, matroids and elsewhere~-- may skip ahead to Section~\ref{sec:applications} without loss.\looseness=-1

Our motivation for introducing the notion of a profile was to find a concept, as general as possible, that captures the idea of identifying highly connected parts in a discrete structure by orienting its low-order separations consistently towards~it. While initially these orientations might have been induced by some concrete highly connected structure such as a $k$-block, and receive their consistency from there, the idea was that they might ultimately be thought of as a highly connected substructure in their own right, thus allowing us to treat these in more general contexts such as abstract separation systems.

The key to this, then, lies in finding an abstract notion of consistency to make this work. Our formal definition of consistency, that two separations should not point away from each other, is just a minimum requirement one would naturally make. Conditions (P) and~(R) go a little further but are still fairly general; recall that tangles satisfy both. As shown in~\cite{confing}, there is little hope of weakening (R) and retaining a tangle-tree theorem. But maybe we can weaken (P) a little, while still keeping it narrow enough that the `weak profiles' it then defines can still be distinguished in a tree-like way?

There is also the question of just how natural is the notion of a profile. Condition~(P) seems very natural as a requirement of consistency: if $\vr$ and $\vs$ both point to some substructure~$X$ then so should $\vt := \vr\lor\vs$ if $t$ is oriented at all, provided the universe in which $\vr\lor\vs$ is taken is dense enough%
   \COMMENT{}
   that there is `no room for~$X$' between $\vr$ and $\vs$ on the one hand and $\rv\land\sv = \tv$ on the other (cf.\ Example~\ref{ex:resolution}). 
   And surely, then, every $\vt\le\vr\lor\vs$ should also lie in such a profile~$P$, as soon as~$t$~-- rather than $\vr\lor\vs$~-- is oriented at all, even if $\vr\lor\vs$ is not. Our current definition ensures this (by consistency) only if $\vr\lor\vs$ itself is oriented too, and hence lies in~$P$.

Let us call a consistent orientation $P$ of a separation system $(\vS,\le\,,\!{}^*)$ in some universe $\vU = (\vU,\le\,,\!{}^*,\vee,\wedge)$ a {\em strong profile\/} if it satisfies this strengthening of~(P):\looseness=-1
 $$\text{For all $\vr,\vs \in P$ and $\vt\le\vr\lor\vs$ the separation $\tv$ is not in $P$.}\eqno\rm (P^+)$$

Surprisingly, perhaps, if $\vU$ is a distributive as a lattice%
   \COMMENT{}
   and $\vS$ is {\em submodular\/} in the sense that, whenever $\vr,\vs\in\vS$, either $\vr\lor\vs\in\vS$ or $\vr\land\vs\in\vS$~\cite{AbstractSepSys}, then all its profiles are strong:

\begin{thm}\label{P3} {\rm\cite[Theorem~1]{EberenzMaster}}
If $\vU$ is distributive and $\vS$ is submodular, then an orientation of $S$ satisfies {\rm (P$^+$)} if and only if it is consistent and satisfies~{\rm (P)}.
\end{thm}%
   \COMMENT{}

\noindent
We shall apply Theorem~\ref{P3} in Section~\ref{subsec:matroids} to show that, with one simple exception, all profiles in a matroid are in fact tangles.%
   \COMMENT{}

\medbreak

So strengthening profiles in this way yields little new~-- and anyway, our aim was to weaken~(P), not strengthen it. But we can use the same idea to do that too, and in a way that also will make the notion of a profile more symmetric.

Condition~(P$^+$) says that $P$ contains no triple $(\vrone,\vrtwo,\vrthree)$ such that $\rvthree\le\vrone\lor\vrtwo$. To arrive at a weaker notion, let us ban a triple $(\vrone,\vrtwo,\vrthree)$ only if it satisfies this for every permutation. More precicely, let us say that $\vrone,\vrtwo,\vrthree\in\vU$ form a {\em bad triple\/} if
 $$ \rvk\le \vri\lor\vrj\quad\text{whenever}\quad \{i,j,k\} = \{1,2,3\}$$
 (Figure~\ref{fig:badtriple}), and call a consistent orientation of $S\sub U$ a {\em weak $S$-profile\/} if it contains no bad triple.%
   \footnote{One can, of course, construct separation systems with profiles that are not weak profiles in this sense.%
   \COMMENT{}
   However, there are no natural ones: by Theorem~\ref{P3}, profiles of submodular separation systems in distributive universes~-- which covers all natural examples~-- are also weak profiles.}%
   \COMMENT{}

   \begin{figure}[htpb]
\centering\vskip-6pt\vskip0pt
        \includegraphics{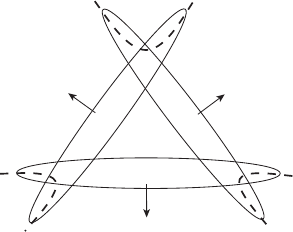}
        \caption{A bad triple of separations}
   \label{fig:badtriple}\vskip-3pt\vskip0pt
   \end{figure}

However, Eberenz~\cite{EberenzMaster} showed that weak profiles can no longer be separated in a tree-like way, as is our aim in this paper. Indeed, we have the following example:

\begin{Bsp} \label{exmpl:awful}
    Let $S$ consists of two crossing separations $r, s$ and their four corner separations. Let all these six separations have the same order. Let $O\sub\vS$ be the set the four corner separations oriented `towards the centre', that is, the set of separations of the form $\vr\land\vs$, one for every choice of orientations $\vr$ of~$r$ and $\vs$ of~$s$. Each of the four possible orientations $\vr$ and~$\vs$ of $r$ and~$s$ extends $O$ to an orientation $P=P_{\!\vr\!\!,\!\vs}$ of~$S$. If~$\vS$, together with a greatest element~$1$ and a least element~0, is the entire universe in which $\lor$ and~$\land$ are defined, then each of these $P$ will contain a bad triple of the form $(\vr,\vs,\rv\land\sv)$, so none of them is a weak profile.%
   \COMMENT{}

However, let us choose a graph $G$ so that $S$ becomes its set of proper 3-separations, and let $\lor$ and~$\land$ be defined as usual in the universe of all separations of~$G$. Then the four orientations $P$ of $S$ defined above are weak 3-profiles. To see this, notice that if $\vec x$ is a corner separation and $\vec y$ is any of the other separations in~$P$, then $\vx\lor\vy$ has order at least~4 unless $\vec x\le\vec y$, and $r$ has no orientation $\vr\le\vx\lor\vy$ (and likewise for~$s$).%
   \COMMENT{}
   Here, $\vx\lor\vy\notin\vS$, but allowing this, while still banning $\rv\in\vS$ from~$P$ if $\vr\le\vx\lor\vy$ and $\vx,\vy\in P$, was the whole point of introducing bad triples and weak profiles.\looseness=-1

Now we cannot distinguish these four weak profiles by a nested set of separations. Indeed, $r$~and~$s$ are the only separations that can distinguish any of them, we cannot have both $r$ and~$s$ in a nested set, but one of them will not distinguish the four weak profiles pairwise.%
   \COMMENT{}
   \end{Bsp}

\section{Applications to graphs and matroids}\label{sec:applications}

In this section we apply the results of Section~\ref{secTT} to graphs and matroids, concluding with
the proofs of the theorems stated in the introduction. We use the terms `block', `profile' and `tangle' for $k$-blocks, $k$-profiles and $k$-tangles with any~$k$. By a profile {\em in\/} a graph or matroid we mean a profile in the universe of all its separations.

\subsection{Canonical tangle-tree theorems in graphs}\label{subsec:graphs}

Let $G=(V,E)$ be a finite graph, and let $\vU = (\vU,\le\,,\!{}^*,\vee,\wedge,|\ |)$ be the universe of all its oriented separations~$(A,B)$ with the usual submodular order function of $|(A,B)| := |A,B| := |A\cap B|$. We think of $\{A,B\}$ as the corresponding unoriented separation, which officially in our terminology is the set~$\{(A,B),(B,A)\}$.

Given a $k$-block $b$ in~$G$, let us write
$$P_k (b)\defgl \{\, (A,B)  \in \vS_k \mid b \subseteq B\,\}$$
for the orientation of $S_k$ it induces. In the introduction, we called this the `profile of~$b$'. Our later definition of abstract profiles bears this out:

\begin{Lem}\label{blockprofiles}
   $P_k (b)$ is a regular $k$-profile.\qed
\end{Lem}
   \COMMENT{}

\noindent
We shall call the profiles of the form $P_k(b)$ the ($k$-){\em block profiles\/} in~$G$.

Notice that a separation $(A,B)\in S_k$ distinguishes two $k$-blocks, in the original sense that one lies in~$A$ and the other in~$B$, if and only if $(A,B)$ distinguishes the $k$-profiles they induce. From now on, we shall use the term `distinguish' formally only for profiles, and say that a separation {\em distinguishes\/} two blocks, say, if it distinguishes their profiles. In this way, we can now also speak of separations distinguishing blocks from other profiles such as tangles.

Similarly, we shall call a block~$b$ \emph{robust} if the profiles%
   \footnote{Note that $b$ can be a $k$-block for several~$k$. If the corresponding sets $S_k$ are distinct, then so are the $k$-profiles $P_k(b)$ that $b$ induces, though clearly $P_\ell (b)\sub P_k (b)$ for $\ell < k$.}
   it induces are robust, and similarly for sets of blocks.%
   \COMMENT{}

A {\em tangle of order~$k$} in~$G$, or {\em $k$-tangle\/}, is an orientation~$\theta$ of~$S_k$ such that
 $$G[A_1] \cup G[A_2] \cup G[A_3]\ne G\eqno\rm(T)$$
 for all triples $(A_1,B_1), (A_2,B_2), (A_3,B_3)\in\theta$. These $(A_i,B_i)$ need not be distinct, so tangles are consistent. In fact, the following is immediate from the definitions:

\begin{Lem}\label{lem:K-robust}
   Every tangle in~$G$ is a robust regular profile.\qed
\end{Lem}
   \COMMENT{}

In order to apply Theorem~\ref{thm2} to prove Theorem~3, we need to show that every tree set of separations in a graph $G$ is induced by a suitable \td\ of~$G$. The following lemma, which was proved for regular tree sets in~\cite{confing} and for arbitrary tree sets in~\cite{TreeSets},%
   \footnote{Part~(ii) of Lemma~\ref{lem:canon} was not stated in the sources cited, but the proofs given  establish~it.}
    achieves this:

\begin{Lem}\emph{}\label{lem:canon}
   Let $T$ be a tree set%
   \COMMENT{}
   of separations of~$G$. Then $G$ has a \td\ $(\mathcal{T},\mathcal{V})$ such that:\vskip-\medskipamount\vskip0pt
    \begin{enumerate}[label = {\rm(\roman*)}]\itemsep0pt
      \item $T$ is precisely the set of separations of $G$ associated with the edges of~$\mathcal T\!$;
      \item if $T$ is invariant under the automorphisms of $G$, then $(\mathcal{T},\mathcal{V})$ is 
      canonical.%
   \footnote{Canonical \td s of graphs are defined in the Introduction.}
   \end{enumerate}
\end{Lem}

Note that we have an immediate proof of Theorem~\ref{fixedk} now. Indeed, since $k$-profiles are $n$-robust for all $n\le k$,%
   \COMMENT{}
   the set $\P$ of all $k$-block profiles and $k$-tangles in a graph~$G$, for $k$ fixed, is always robust. Theorem~\ref{thm2} therefore yields a regular%
   \COMMENT{}
   and invariant tree set $T$ distinguishing all its $k$-blocks and $k$-tangles, which by Lemma~\ref{lem:canon} yields the desired \td\ of~$G$.

Theorem~\ref{generalk} follows in a similar way from Corollary~\ref{cor3}. This takes care of the fact that, when $k$ is variable, profiles can induce each other:

\begin{thmgeneralk}\em
Every finite graph $G$ has a canonical \td{} that efficiently distinguishes all its distinguishable tangles and robust blocks.%
   \COMMENT{}
\end{thmgeneralk}

\begin{proof}
By Lemmas \ref{blockprofiles} and~\ref{lem:K-robust}, all the profiles in~$G$ that are tangles or induced by a robust block are robust and regular profiles. By Corollary~\ref{cor3} there is a canonical regular tree set $T$ that efficiently distinguishes every two of them. Lemma~\ref{lem:canon} turns $T$ into a \td\ of~$G$.
   \end{proof}

Let us conclude this section with two observations about profiles in graphs.
The first is a kind of `inverse consistency' phenomenon. If $P$ is a $k$-profile%
   \COMMENT{}
   in a graph, $(A,B)\in P$, and $\{A',B'\}\in S_k$ is such that $A'\sub A$ and $B'\supe B$, then also $(A',B')\in P$, by the consistency of~$P$ and $(A',B')\le (A,B)$.

But now suppose that, instead, $A'\supe A$ and $B'\sub B$. Then $(A,B)\le (A',B')$, but still we get $(A',B')\in P$ if $P$ is regular%
    \footnote{Regularity is not a severe restriction. We shall see in a moment that $k$-profiles in graphs are all regular for $k\ge 3$, and for $k=1,2$ the few irregular $k$-profiles are exactly understood.}%
   \COMMENT{}
    and $\{A',B'\}$ does not differ from $\{A,B\}$ too much:

\begin{Prop}\label{inverseconsistent}
   Let $P$ be a regular $k$-profile in a graph~$G$, and let $(A,B) \in P$. Let $A'\supe A$ and $B'\sub B$ be such that both $\{A,B'\}$ and $\{A',B'\}$ lie in~$S_k$.%
   \COMMENT{}
   Then $(A^\prime,B^\prime)\in P$.
\end{Prop}

\begin{proof}
Let us first prove the assertion for $A'=A$, i.e., show that ${(A,B')\in P}$.%
   \COMMENT{}
We show that if this fails%
   \COMMENT{}
   then $\{(A,B), (B',A), (B\cap A, A\cup B')\}\sub P$, in violation of~(P).

We have $\{A,B'\}\in S_k$ by assumption, so if $(A,B')$ does not lie in~$P$ as claimed, then $(B',A)$ does. And since $A\cup B' = V$, we have $\{B\cap A, A\cup B'\}\in S_k$.%
   \COMMENT{}
  Since $P$ is regular, it must contain the small separation $(B\cap A, A\cup B')$.

The case of $A'\supsetneq A$ can now be derived from this as follows. If ${(A',B')\notin P}$ then $(B',A')\in P$, since $\{A',B'\}\in S_k$. Then $(B',A)\in P$ by the above case applied to~$(B',A')$, since $A\sub A'$. But this contradicts the fact that, as established above, $(A,B')\in P$.
\end{proof}

Our second observation is that graphs can have profiles that are neither block profiles nor tangles:

\begin{Bsp}\label{Joh} (Carmesin, personal communication 2014)
The graph in Figure~\ref{pic:Johannes} has a $5$-profile that orients both its 4-separations towards the middle. This profile is neither a 5-tangle nor induced by a 5-block.
\end{Bsp}

\pdffig{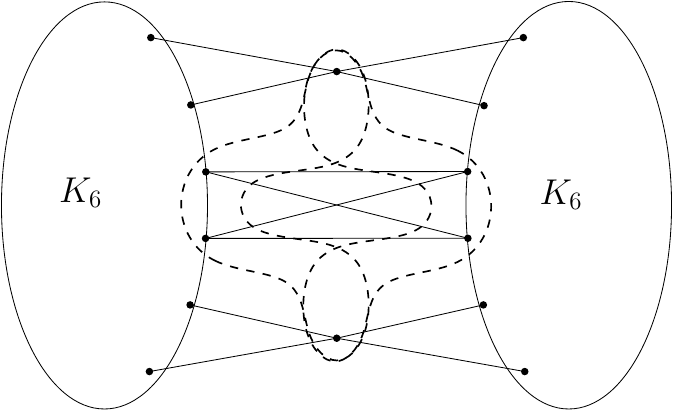}{A 5-profile that is neither a 5-tangle nor induced by a 5-block.}
{0.5}{h}

Example~\ref{Joh} makes it desirable to have a tangle-tree theorem for all the profiles in a graph, not just those that are tangles or block profiles. Corollary~\ref{cor3} yields this too, even without a need to require regularity.

To see this, we need some observations from~\cite{ProfileDuality}. First, that all $k$-profiles in a graph $G=(V,E)$ are regular as soon as $k>2$. For $k=1$ there can be exactly one irregular 1-profile, the set~$\{(V,\emptyset)\}$, and only if $G$ is connected. For $k=2$ there can be irregular profiles, but these can be described precisely: they are precisely the 2-profiles the form
 $$P_v = \{\,(A,B)\in\vS_2\mid v\in B\text{ and } (A,B)\ne (\{v\},V)\,\}$$
where $v\in V$ is not a cutvertex of~$G$. Let us call such profiles in graphs {\em principal\/}. Thus, all profiles in graphs other than those 2-profiles~$P_v$ are non-principal.

\begin{thm}
Every finite graph has a canonical \td\ that distinguishes all its distinguishable robust non-principal profiles.
\end{thm}

\begin{proof}
By Corollary~\ref{cor3} there is a canonical regular tree set $T$ that distinguishes all the distinguishable, robust, non-principal $k$-profiles in~$G$ for $k>1$ as well as for $k=1$ if they are regular, since all these profiles for $k>1$ are regular~\cite{ProfileDuality}. Lemma~\ref{lem:canon} turns $T$ into a \td\ $(\mathcal T,\V)$ that distinguishes all these profiles, with $\V = (V_t)_{t\in\mathcal T}$ say. If $G$ is disconnected, then these are all its robust non-principal profiles, and the proof is complete.

If $G$ is connected, it als has the irregular non-principal 1-profile $\{(V,\emptyset)\}$. To accommodate this, we add a new empty part~$V_t$ to our decomposition, joining $t$ to some node $z$ of~$\mathcal T$ by an edge associated with the separation~$\{\emptyset,V\}$. This separation will distinguish $\{(V,\emptyset)\}$ from all the other profiles in~$G$, since these contain $(\emptyset,V)$ by consistency.%
   \COMMENT{}

The only problem is how to choose~$z$ so that the \td\ remains canonical. If $\mathcal T$ has a central node, we choose this as~$z$. If not, it has a central edge~$xy$. Subdivide $xy$ by a new node~$z$, with $V_z := V_x\cap V_y$. Then every automorphism of~$\mathcal T$ is, or extends to, an automorphism of this (possibly modified) tree~$\mathcal T'$ that fixes~$z$. In particular, the automorphisms of~$G$, which act on~$\mathcal T$ by Lemma~\ref{lem:canon}, still act on~$\mathcal T'$. Now join $t$ to~$z$. Since all automorphisms of $G$ fix $V_t = \emptyset$, the extended \td\ will again be canonical.%
   \COMMENT{}
\end{proof}

\subsection{The canonical tangle-tree theorem for matroids}\label{subsec:matroids}

Let $M$ be a finite matroid with ground set~$E$. Let $\vU = (\vU,\le\,,\!{}^*,\vee,\wedge,|\ |)$ be the universe of the oriented bipartitions of~$E$, where $\lor$ and $\land$ are defined as for set separations in Section~\ref{subsec:SeparationSystems}, and $|\ \ |$ is the usual submodular order function given by
 $$|\{A,B\}| = |(A,B)| = r(A) + r(B) - r(E) +1,$$ 
 where $r$ is the rank function of~$M$. As before, we write $\vS{}_k$ for the separation system of all the separations in~$\vU$ of order less than~$k$.

Tangles in matroids were introduced implicitly by Robertson and Seymour~\cite{GMX}, and explicitly by Geelen, Gerards, Robertson and Whittle~\cite{BranchDecMatroids}. A $k$-{\em tangle\/} in~$M$ is an orientation of~$S_k$ that has no element of the form $(E\sm\{e\},e)$ and no subset of the form
$$\{(A_1,B_1),(A_2,B_2),(A_3,B_3)\}\text{ with }  A_1\cup A_2\cup A_3 = E.\eqno{\rm (T_M)}$$
 As with graphs, the separations $(A_1,B_1), (A_2,B_2), (A_3,B_3)$ above need not be distinct. In particular, tangles are consistent.%
   \COMMENT{}
   Since they obviously also satisfy~(P), they are in fact profiles, indeed robust and regular profiles.%
   \COMMENT{}

Unlike for graphs, however, the tangles in a matroid are essentially all its profiles. Let us call a profile $P$ in~$M$ {\em principal\/} if there exists an $e\in E$ such that $(A,B)\in P$ if and only if $e\in B$, for all bipartitions $\{A,B\}$ oriented by~$P$.

\begin{Lem}
For every integer $k\ge 1$, the $k$-tangles in a matroid are precisely its non-principal $k$-profiles.
\end{Lem}

\begin{proof}
We have seen that matroid tangles are profiles, and they are non-principal by definition. Conversely note that, by consistency, a profile is principal as soon as it contains a separation of the form $(E\sm\{e\},\{e\})$.%
   \COMMENT{}
   Hence a non-principal profile contains no such separation, and thus satisfies the tangle axiom requiring this. And by Theorem~\ref{P3} it satisfies~(P$^+$), which immediately implies~${\rm (T_M)}$.%
   \COMMENT{}
\end{proof}

A \emph{\td} of~$M$ is a pair $(T,\V)$ where $T$ is a tree and $\V = (\,V_t\mid t\in V(T)\,)$ is a partition of~$E$.
As in graphs, we say $(\mathcal{T},\mathcal{V})$ \emph{(efficiently) distinguishes} two 
tangles in~$M$ if the set of separations of~$M$ that are associated with the edges of~$T$~-- defined as for graphs~-- distinguishes these tangles (efficiently). Lemmas \ref{lem:K-robust} and~\ref{lem:canon} hold for \td{}s of matroids too, with the same easy proofs.

$\!\!$Theorem~\ref{matroids} follows from Corollary~\ref{cor3} much as Theorem~\ref{generalk} did:

\begin{thmmatroids}\em
   Every finite matroid has a canonical \td{} which
   efficient\-ly distinguishes all its distinguishable tangles.
\end{thmmatroids}

\begin{proof}
By Lemma~\ref{lem:K-robust}, the tangles in~$M$ are robust and regular profiles. By Corollary~\ref{cor3} there is a canonical regular tree set $T$ that efficiently distinguishes every two of them. Lemma~\ref{lem:canon} turns $T$ into a \td\ of~$M$.
\end{proof}

\section{$\!$Applications in data analysis and optimization}\label{sec:furtherapplications}

\subsection{Applying profiles to cluster analysis}\label{subsec:clusters}

Both the canonical and the non-canonical tangle-tree theorem lend themselves to the analysis of big data sets when clusters in these are interpreted as profiles. This approach to clusters has an important advantage over more traditional ways of identifying clusters: real-world clusters tend to be fuzzy, and profiles can capture them despite their fuzziness; cf.\ the grid example from the introduction.

Let $D$ be a (large data) set with at least two elements,%
   \COMMENT{}
   and let $U$ be a set of separations $\{A,B\}$ of~$D$ whose orientations $(A,B)$ form a universe~$\vec U$ of set separations as defined in Section~\ref{subsec:basics}.%
   \footnote{We allow $A$ and $B$ to be empty, since $\{\emptyset,D\}\in U$ is needed for $\vec U$ to be a universe.}
   Let us call $D$ {\em submodular\/} if it comes with a fixed submodular such universe~$\vU$, one with a submodular order function $\{A,B\}\mapsto |A,B|$ on~$U$.

How exactly this order function should be defined will depend both on the type of data that $D$ represents and on the envisaged type of clustering. See~\cite{MonaLisa} for examples of how this might be done when $D$ is the set of pixels of an image, $U$~consists of just the bipartitions of~$D$, and the clusters to be captured are the natural regions of this image such as a nose, or a cheek, in a portrait.

If $U$ consists of just the bipartitions of~$D$, then all profiles in~$\vec U$ are regular, because the only small such separation of~$D$ is $(\emptyset,D)$,
   which is in fact trivial since $|D|\ge 2$.%
   \COMMENT{}
   And they are all robust; see Footnote~\ref{robustfootnote}. Since every $d\in D$ induces a profile $P_d := \{\,(A,B)\mid d\in B\,\}$ of~$U$ (just choose $k$ large enough),%
   \COMMENT{}
   any tree set of separations that distinguishes all the profiles in~$\vec U$ will contain, for every pair of distinct data $a,b\in D$, a~bipartition $\{A,B\}$ of~$D$ such that $a\in A$ and $b\in B$.

Let us call the robust and regular profiles in~$U$ the {\em clusters\/} of~$D$. Theorem~\ref{clusters} then becomes a consequence of Theorem~\ref{thm2}, via Corollary~\ref{cor3}:

\begin{thmclusters}
Every submodular data set has a canonical regular tree set of separations which efficiently distinguishes all its distinguishable clusters.%
   \COMMENT{}
   \qed
\end{thmclusters}

\subsection{Edge-tangles and Gomory-Hu trees}\label{edge-tangles}

In this section we illustrate Theorem~\ref{thm2min} by applying it to some specific profiles, sometimes called `edge-tangles'.%
   \COMMENT{}
   We reobtain a classical theorem of Gomory and Hu~\cite{GomoryHu}, which asserts the existence of certain trees displaying optimal cuts in a graph. Our point is not that obtaining these trees via Theorem~\ref{thm2min} is simpler than their elementary existence proof, although that is not exactly short either.  But seeing what Theorem~\ref{thm2min} does in this special case can illustrate its power: it shows, in a familiar setting, what it can do more generally with profiles that arise in other contexts, such as~\cite{ProfileDuality}.

Consider a finite graph $G=(V,E)$ with a positive real-valued weight function~$g$ on the edges. A~tree $T$ on~$V\!$, not necessarily a subgraph of~$G$, is a {\it Gomory-Hu tree\/} for $G$ and~$g$ if for all distinct $u,v\in V$ there is an edge on the $u$--$v$ path in~$T$ whose fundamental cut in~$G$%
   \COMMENT{}
   has minimum $g$-weight among all $u$--$v$ cuts of~$G$. (For details, see Frank~\cite[Ch.\,7.2.2]{FrankBook}.)

\begin{thm} {\rm (Gomory \& Hu, 1961)}\\
Every finite graph with positive real edge weights has a Gomory-Hu tree.
\end{thm}

\begin{proof}
To deduce this theorem from Theorem~\ref{thm2min}, consider the universe $\vU$ of all set separations of~$V\!$ that are bipartitions of~$V\!$ with the order function $|A,B| := \sum\>\{\,g(ab)\mid a\in A,\> b\in B,\> ab\in E(G)\,\}$, which is clearly submodular. Every vertex $v\in V$ induces a robust regular profile $P_v := \{\,(A,B)\mid v\in B\,\}$ of~$U$, and $v\mapsto P_v$ is a bijection from~$V\!$ to the set $\P$ of all these profiles of~$U$.%
   \COMMENT{}

Let $T'\sub U$ be the tree set which Theorem~\ref{thm2min} finds for~$\P$.
For every $v\in V\!$, the profile $P_v$ of~$U\!$ induces a profile~$t_v$ of $T'\sub U$, and $v\mapsto P_v\mapsto t_v$ is a bijection from $V\!$ to the set $\O$ of all the profiles of~$T'$. (Since $T'$ is a tree set, its profiles are just its consistent orientations.)%
   \COMMENT{}

Every $s\in T'$ induces a bipartition of~$\O$ into the set $\O_{\!\vs}$ of profiles that orient $s$ as~$\vs$ and the set~$\O_{\!\sv}$ of profiles that orient $s$ as~$\sv$. And $s$ is itself a bipartition of~$V\!$, say $\vs=(U,W)$. These bipartitions correspond via our bijection~$v\mapsto t_v$:
 $$U = \{\, v\in V\mid \sv\in t_v\,\} = \{\, v\in V\mid t_v\in \O_{\!\sv}\}$$
 $$W = \{\, v\in V\mid \vs\in t_v\,\}= \{\, v\in V\mid t_v\in\O_{\!\vs}\}.$$
In other words, our bijection $v\mapsto t_v$ maps $\vs = (U,W)$ to~$(\O_{\!\sv}, \O_{\!\vs})$.
 
It is not hard to show~\cite{TreeSets} that for every $s\in T'$ there exist unique $t_u,t_v\in\O$ that differ only on~$s$.%
   \COMMENT{}
   Then $T'$ becomes the edge set of a graph~$\T$ on~$\O$ if we let $s$ join $t_u$ to~$t_v$, for every $s\in T'$. This graph~$\T$ is in fact a tree~\cite{TreeSets}. For every edge $s$ of~$\T$, the vertex sets $\O_{\!\vs},\O_{\!\sv}$ induce the two components of the forest $\T-s$, because profiles adjacent as vertices of~$\T-s$ agree on~$s$.%
   \COMMENT{}

Let $T$ be the tree on~$V\!$ that is the image of~$\T$ under our bijection $t_v\mapsto v$ from $\O$ to~$V\!$. Every edge $e$ of~$T$ then arises from an edge $s$ of~$\T$, and the bipartition $\{U,W\}$ of~$V$ defined by the components of $T-e$ corresponds via $v\mapsto t_v$ to the bipartition $\{\O_{\!\sv}, \O_{\!\vs}\}$ of~$\O$ defined by the components of~$\T-s$. As shown earlier, this means that $s = \{U,W\}$.

Let us show that $T$ is a Gomory-Hu tree for $G$ and~$g$. For all distinct $u,v\in T$ the separations in $T'$ that distinguish $t_u$ from~$t_v$ (equivalently: $P_u$ from~$P_v$) are precisely the edges on the $t_u$--$t_v$ path in~$\T$.%
   \COMMENT{}
   One of these, $s$~say, distinguishes $P_u$ from~$P_v$ efficiently, by the choice of~$T'$. By definition of~$T$, the edge $e\in T$ corresponding to~$s$ lies on the $u$--$v$ path in~$T$. The bipartition of $V\!$ defined by the two components of $T-e$%
   \COMMENT{}
   is precisely~$s$. The fact that $s$ distinguishes $P_u$ from~$P_v$ efficiently means that the cut of~$G$ it defines has minimum weight among all the $u$--$v$ cuts of~$G$, as desired.
   \end{proof}

\section{Width duality for profiles in graphs}\label{sec:duality}

All our theorems deal with distinguishing existing profiles in a given universe of separations, but say nothing if such profiles do not exist. For fixed~$k$, for example, we can ask what we can say about the structure of a graph that has no $k$-profile. Such graphs do exist:

\begin{Bsp}
Let $G$ be a graph with a \td\ of width~$<k-1$, one whose parts have order~$<k$. Any $k$-profile $P$ in $G$ will orient all the separations associated with an edge of the decomposition tree~$T$, and will thus induce a consistent orientation of~$E(T)$. This orientation will points towards some node~$t$ of~$T$.

Now consider a separation of $G$ associated with an edge $t't$ at~$t$, say $(A,B)$, with $B$ containing the part~$V_t$ corresponding to~$t$. Let $A':= A\cup V_t$. Then $(A',B)$  is still in~$\vS_k$, since $|A'\cap B| = |V_t| < k$, and $(A',B)\in P$ by Proposition~\ref{inverseconsistent}.%
   \COMMENT{}

If $T$ has another edge at~$t$, say $t''t$, do the same to obtain $(A'',C)\in P$. Then, by~(P), also $(A'\cup A'', B\cap C)\in P$; note that this separation too has order $|V_t|<k$. 

Repeating this step for all the edges of $T$ at~$t$, we eventually arrive at $(V,V_t)\in P$.%
   \COMMENT{}
   If $|V_t|\le k-2$, which we can easily ensure, this contradicts the consistency of~$P$, since $(V,V_t)$ will then be co-trivial: just add a vertex to its $V_t$-side to obtain a larger separation than $(V_t,V)$ that is still small.
\end{Bsp} 

For tangles, Robertson and Seymour~\cite{GMX} address the problem of characterizing the structure of graphs without a tangle of given order: they show that a graph either has a $k$-tangle or a \td%
   \footnote{The `branch-decompositions' used in~\cite{GMX} are easily translated into \td s~\cite{TangleTreeGraphsMatroids}.}
   that is clearly incompatible with the existence of a $k$-tangle, and which thus witnesses its non-existence. Similar duality theorems of this type were obtained in~\cite{TangleTreeGraphsMatroids} for other tangle-like structures, consistent orientations of separation systems avoiding some given set~$\mathcal F$ of `forbidden subsets'. (For classical tangles, $\mathcal F$~would consist of the triples $\{(A_1,B_1), (A_2,B_2), (A_3,B_3)\}$ satisfying~(T).) All these are applications of a fundamental such duality theorem for abstract separation systems~\cite{TangleTreeAbstract}.

The sets~$\mathcal F$ for which such duality theorems are proved in~\cite{TangleTreeAbstract} have to consist of `stars': sets of separations all pointing towards each other. For classical tangles this does not matter: one easily shows that a consistent orientation of~$S_k$, say, avoids all the triples as in~(T) as soon as it avoids the stars satisfying~(T). For profiles, however, this is not the case~-- which is why \cite{TangleTreeGraphsMatroids} has no duality theorem for profiles.

However, it was shown in~\cite{ProfileDuality} that profiles can be characterized by a forbidden set $\mathcal F$ of stars of separations after all, and thus have a duality theorem. Indeed, for each~$k$ there is a set ${\mathcal T}_{\mathcal P}(k)$ of \td s which graphs with a $k$-profile cannot have, and which graphs without a $k$-profile must have:

\begin{thm}\label{t:prof} {\rm\cite{ProfileDuality}}
For every finite graph $G$ and every integer $k > 0$ exactly one of the following statements holds:
\begin{itemize}\itemsep=0pt\vskip-\smallskipamount\vskip0pt
\item $G$ has a $k$-profile;
\item $G$ has a \td\ in~${\mathcal T}_{\mathcal P}(k)$.
\end{itemize}
\end{thm}

\noindent
   A similar duality theorem holds more generally for regular profiles of abstract separation systems~\cite[Theorem~12]{ProfileDuality}.

\section{An open problem for tangles}\label{sec:problems}

Unlike blocks and tangles, profiles can exist in abstract separation systems that need not even consist of separations of sets. This motivates the notion of a profile independently of their examples in graphs or matroids.

But for profiles in graphs, it is meaningful to ask just how different arbitrary profiles can be from block profiles: are they always induced by some fixed set of vertices in the same way as a block profile $P_k(b)$ is induced by~$b$?

This problem appears to be open even for tangles:

\begin{Problem}
   Given a $k$-tangle $\theta$ in a graph~$G$, is there always a set $X$ of vertices such that a separation $(A,B)$%
   \COMMENT{}
   of order~$<k$ lies in~$\theta$ if and only $|A\cap X| < |B\cap X|$?%
   \COMMENT{}
\end{Problem}

\COMMENT{}

\bibliography{collective}{}
\bibliographystyle{abbrv}

\newpage

\thispagestyle{empty}

\end{document}